\documentclass[11pt,a4paper]{amsart}
\usepackage{mathtools,amssymb,amsmath,amsopn,amsthm,graphicx,MnSymbol,centernot,stmaryrd,array}
\usepackage{tikz}
\usetikzlibrary{calc}
\usepackage{tikz-cd}
\usepackage{enumitem}
\usepackage{amsaddr}
\usepackage{etoolbox}
\patchcmd{\subsection}{-.5em}{.5em}{}{}
\usetikzlibrary{matrix}
\begin{document}

\newtheorem{definition}{Definition}[section]
\newtheorem{definitions}[definition]{Definitions}
\newtheorem{deflem}[definition]{Definition and Lemma}
\newtheorem{lemma}[definition]{Lemma}
\newtheorem{proposition}[definition]{Proposition}
\newtheorem{theorem}[definition]{Theorem}
\newtheorem{corollary}[definition]{Corollary}
\newtheorem{algo}[definition]{Algorithm}
\theoremstyle{remark}
\newtheorem{rmk}[definition]{Remark}
\theoremstyle{remark}
\newtheorem{remarks}[definition]{Remarks}
\theoremstyle{remark}
\newtheorem{notation}[definition]{Notation}
\newtheorem{assumption}[definition]{Assumption}
\theoremstyle{remark}
\newtheorem{example}[definition]{Example}
\theoremstyle{remark}
\newtheorem{examples}[definition]{Examples}
\theoremstyle{remark}
\newtheorem{dgram}[definition]{Diagram}
\theoremstyle{remark}
\newtheorem{fact}[definition]{Fact}
\theoremstyle{remark}
\newtheorem{illust}[definition]{Illustration}
\theoremstyle{remark}
\newtheorem{que}[definition]{Question}
\theoremstyle{definition}
\newtheorem{conj}[definition]{Conjecture}
\newtheorem{scho}[definition]{Scholium}
\newtheorem{por}[definition]{Porism}
\DeclarePairedDelimiter\floor{\lfloor}{\rfloor}

\renewenvironment{proof}{\noindent {\bf{Proof.}}}{\hspace*{3mm}{$\Box$}{\vspace{9pt}}}
\author[Sardar]{Shantanu Sardar}
\address{CEMIM, FCEyN \\Universidad Nacional de Mar del Plata\\ CONICET, Argentina}
\email{shantanusardar17@gmail.com}
\title{{Super-decomposable pure-injective modules over some Jacobian algebras}}
\subjclass[2020]{}

\begin{abstract}
Existence of superdecomposable pure-injective modules reflects complexity in the category of finite-dimensional representations over an algebra. Such an existence occurs when an algebra is non-domestic; a conjecture due to M. Prest. G. Puniski confirms the conjecture for non-domestic string algebras. Geiß, Labardini-Fragoso and Schr\"oer show that every Jacobian algebra associated with a triangulation of a closed surface with marked points is finite-dimensional and tame. We show that, excluding only the case of a sphere with four (or fewer) punctures, there exists a special family of pointed modules, called an independent pair of dense chains of pointed modules. In the process, we show the existence of such an independent pair in a non-domestic skew-gentle algebra and (skew) Brauer graph algebras by showing that the Galois semi-covering functor and trivial extension preserve such pairs. Then it follows from a result of M. Ziegler that there exists a superdecomposable pure-injective module if the algebraically closed field is countable.
\end{abstract}

\maketitle
\section{Introduction}
Classification of finite-dimensional associative algebras $\Lambda$ has been an ongoing project for a long time among the representation theorists. Roughly speaking, $\Lambda$ is tame if a description of all finite-dimensional $\Lambda$-modules is available. Otherwise, $\Lambda$ is wild. Drozd's theorem in \cite{Dr79} states a well-known dichotomy that every finite-dimensional algebra over an algebraically closed field $K$ is either tame or wild (see definitions in \cite{SiSk07}) but not both. The class of tame algebras may be stratified with respect to the growth of a tame algebra, thanks to A. Skowroński. This stratification distinguishes the classes of algebras of domestic, linear, polynomial and exponential growth in \cite{Sk90}. The growth of a tame algebra $\Lambda$ reflects the complexity of the category $\Lambda$-mod of all finite-dimensional left $\Lambda$-modules.

An independent pair of dense chains of pointed $\Lambda$-modules allows us to formulate a handy sufficient condition for the existence of a super-decomposable (i.e., without indecomposable direct summands) pure-injective module. M. Ziegler \cite{Zieg84} established a fundamental existing criterion for such modules, which states that a countable ring $\Lambda$ has a super-decomposable pure-injective module if and only if the width of the lattice of all pp-formulae is undefined. The later statement can be formulated in terms of the width of the lattice of all pointed finitely presented $\Lambda$-modules. The existence of such modules could assist with the classification problem as described in the following conjectures due to Prest. The first one is that $\Lambda$ is tame if and only if it does not possess a super decomposable pure-injective module. The second conjecture \cite{Pr13} states that an algebra $\Lambda$ is of domestic representation type iff there is no super-decomposable pure-injective module. G. Puniski refutes the first conjecture for non-domestic string algebras \cite{Puni04}. Here, we show the existence of such a module over several tame algebras-some Jacobian algebras, non-domestic skew-gentle algebras, their trivial extension algebras and (skew-) Brauer graph algebras over a countable field.
 
\noindent{\textbf{Some cases with $\mathrm{KG}(\Lambda)=\infty$:}} Here is a list of algebras where the width of the lattice of all pointed modules is infinite. As a result, their KG dimension is undefined. M. Prest \cite{Pr88} proves that super-decomposable pure-injective modules exist over strictly wild algebras, whereas the existence is shown in wild algebras by L. Gregory and M. Prest \cite{GP}. In \cite{Puni04}, G. Puninski shows that super-decomposable pure-injectives exist over string algebras of non-polynomial growth if the base field is countable. Later, he provide a explicit construction of such a module in \cite{Pu08} for some string algebra of non-polynomial growth over an arbitrary field. R. Harland proves in \cite{Ha11} that such a module exists for tubular algebras. S. Kasjan and G. Pastuszak show in \cite{KaPa14} that there are super-decomposable pure-injective modules over all strongly simply connected $K$-algebras of non-polynomial growth. The main result of \cite{Pa14} asserts that if $\Lambda$ is a strongly simply connected $K$-algebra and $K$ is countable, then $\Lambda$ is of non-domestic type if and only if there exists a super-decomposable pure-injective $\Lambda$-module. The representation-infinite domestic standard self-injective algebras over algebraically closed fields do not have super-decomposable pure-injective modules (see \cite[Theorem 8.3]{Pa19}). All the known results support the second conjecture of Prest concerning super-decomposable pure-injective modules. Here, we enlarge the list by showing the existence of an independent pair of dense chains of pointed modules and a super-decomposable pure-injective module over non-domestic skew-gentle algebras and some Jacobian algebras. 

A potential $W$ for a quiver $Q$ is, roughly speaking, a linear combination of cyclic paths in the complete path algebra $K\langle \langle Q\rangle \rangle$. The Jacobian algebra $P(Q, W)$ associated with a quiver with a potential $(Q, W)$  is the quotient of the complete path algebra $K\langle \langle Q\rangle \rangle$ modulo the Jacobian ideal $J(W)$. Here, $J(W)$ is the topological closure of the ideal of $K\langle \langle Q\rangle \rangle$ which is generated by the cyclic derivatives of $W$ with respect to the arrows of $Q$. Quivers with potential were introduced in \cite{DWZ08} to construct additive categorifications of cluster algebras with skew-symmetric exchange matrix. For that, the potential for $Q$ must be non-degenerate, i.e., it can be mutated along with the quiver arbitrarily. In \cite{FST08}, the authors introduced, under some mild hypothesis, for each oriented surface with marked points $(S, M)$, a mutation finite cluster algebra with skew-symmetric exchange matrices. More precisely, each triangulation $T$ of $(S, M)$ by tagged arcs corresponds to a cluster and the corresponding exchange matrix is conveniently coded into a quiver $Q(T)$. Labardini-Fragoso in \cite{La09} enhanced this construction by introducing potentials $W(T)$ and showed that these potentials are compatible with mutations. In particular, these potentials are non-degenerate. Ladkani \cite{La12} showed that for surfaces with empty boundary and a triangulation $T$ that has no self-folded triangles, the Jacobian algebra $P(Q(T), W(T))$ is symmetric (and, in particular, finite-dimensional). It follows that for any triangulation $T'$ of a closed surface $(S, M)$, the Jacobian algebra $P(Q(T'), W(T'))$ is weakly symmetric \cite{HI11}. In \cite{GLS16}, the authors use a degeneration argument to show that these algebras are tame. Moreover, Y. Valdivieso-D\'{i}az \cite{Y16} showed that these algebras, except for the case of a sphere with four (or fewer) punctures, exhibit exponential growth. The algebras associated with a sphere with five punctures are associated with a quotient of a skew-gentle algebra, and we produce a `non-symmetric' independent dense pair of pointed modules (See Definition \ref{nonsy}) over the associated gentle algebra (Theorem \ref{NSDP}), which is non-domestic (Proposition \ref{NDSG}), whereas in the rest of the cases, they are presented as the quotient of non-domestic string algebras (Proposition \ref{NDS}). We show the existence of an independent pair of dense chains of pointed modules over a skew group algebra when the given algebra has a non-symmetric one.

\begin{theorem}\label{NSIDPM1}
Suppose $\mathrm{char}(K)\neq 2$ and $F_\lambda: \mathrm{mod}\mbox{-}\Lambda \to \mathrm{mod}\mbox{-}\bar\Lambda$ is a Galois semi-covering and the module $\theta$ is semisimple. 
\begin{enumerate}
\item If $((M_q, \chi_{M_q})_{q\in L_1}, (N_t, \chi_{N_t})_{t\in L_2})$ is a non-symmetric pair of dense chains of pointed modules over $\Lambda$, then $((F_\lambda(M_q), \chi_{F_\lambda(M_q)})_{q\in L_1}; (F_\lambda (N_t); \chi_{F_\lambda(N_t)})_{t\in L_2})$ is an independent pair of dense chains of pointed modules over $\bar{\Lambda}$.

\item If $\overline{(M_p; \chi_{M_p})}_{p\in L_1}$ is a wide poset of $\theta$-pointed modules in $P_\Lambda^\theta$, then $\overline{(F_\lambda(M_p); \chi_{F_\lambda(M_p))}}_{p\in L_1}$ is a wide poset of $\tilde{\theta}$-pointed modules in $P^{\tilde{\theta}}_{\bar{\Lambda}}$.
\end{enumerate}
\end{theorem}
 
As a consequence of the above theorem and the Ziegler criterion stated in Theorem \ref{Sup}, we obtain the following result for skew-gentle algebras.
\begin{theorem}\label{SDSG}
Suppose $\bar{\Lambda}$ is a skew-gentle algebra and $\Lambda$ is the associated gentle algebra. If $\Lambda$ is non-domestic, then $\bar{\Lambda}$ has a superdecomposable module. Moreover, $\mathrm{KG}(\bar\Lambda)$ is infinite.  
\end{theorem}

Moreover, super-decomposable modules exist in the particular class of Jacobian algebras.

\begin{theorem}\label{SDJ}
Assume that $S$ is a closed oriented surface with a non-empty finite collection $M$ of punctures, excluding only the case of a sphere with four (or fewer) punctures. For an arbitrary tagged triangulation $T$ in the Jacobian algebra $\Lambda= P(Q(T), W(T))$, there exists an independent pair of dense chains of pointed modules in $\Lambda$-mod. In consequence, there exists a super-decomposable pure-injective $\Lambda$-module, if the base field $K$ is countable.
\end{theorem} 

We also show that the KG dimension of the above-mentioned Jacobian algebra is infinite.

\begin{theorem}
Assume that $S$ is a closed oriented surface with a non-empty finite collection $M$ of punctures, excluding only the case of a sphere with four (or fewer) punctures. For an arbitrary tagged triangulation $T$, in the Jacobian algebra $\Lambda= P(Q(T), W(T))$, the Krull–Gabriel dimension $\mathrm{KG}(\Lambda)$ of $\Lambda$ is infinite.
\end{theorem} 

We demonstrate the existence of superdecomposable modules over some well-known algebras like the incidence algebra of the Nazarova–Zavadskij poset, diamond algebra, a garland and some surface algebra as an application of Theorem \ref{NSIDPM1}. We show that superdecomposable modules are preserved when we consider the trivial (or relation) extension algebra:

\begin{theorem}\label{TSDM}
Suppose $\Lambda= \mathcal{K}Q_\Lambda/I_\Lambda$ be a finite-dimensional algebra and let $T(\Lambda)= \mathcal{K}Q_{T(\Lambda)}/I_{T(\Lambda)}$ be its trivial extension algebra. Then there exists a superdecomposable module in $T(\Lambda)$ if there exists a superdecomposable module in $\Lambda$.
\end{theorem}

We produce a counter-example (Example \ref{BGE}) to show that the converse of the above theorem is not true in general. The (skew-) Brauer graph algebras (with multiplicity identically one) are obtained as the trivial extension algebras of (skew-) gentle algebras (resp. \cite[Theorem~3.11]{So24} and \cite[Corollary~3.14]{Schrl18}). Brauer graph algebras are symmetric special biserial \cite{Schr99}, and the (skew-) Brauer graph algebras are obtained as the skew group algebras of Brauer graph algebras \cite[Proposition~3.9]{So24}. We end the paper with a discussion on the existence of super-decomposable modules over (skew-) Brauer graph algebras (Remark \ref{SDBGA}). 

We arrange the paper as follows. We start with the basic information on pointed modules over an algebra in Section $(2)$. Theorem \ref{Unw} states that the existence of an independent pair of dense pointed modules ensures the existence of a wide lattice of such modules, which essentially leads to the existence of a superdecomposable module over the algebra (Theorem \ref{Sup}), thanks to Ziegler. Moreover, we study skew group algebras and present some results about the Galois semi-covering functor over the module category from \cite{GoSaTr25}. We show that this functor is indeed exact and faithful by Remark \ref{faex}. We end this section with Proposition \ref{pm}, which states that the semi-covering functor respects the pointed morphisms between pointed modules. In Section $(3)$, we formulate some results from \cite{Puni04} regarding the existence criteria of an independent pair of dense chains of pointed modules over any non-domestic string algebra. We show that if there exists a non-symmetric independent pair of dense chains of pointed modules over an algebra, then we obtain an independent pair of dense chains of pointed modules over its skew group algebra under the Galois semi-covering functor (Theorem \ref{NSIDPM}). Moreover, we find such a non-symmetric pair in any non-domestic gentle algebra associated with a skew-gentle algebra, which produces a superdecomposable module over the skew-gentle algebra, resulting in their KG-dimension being undefined (Theorem \ref{SDSG}). In Section $(4)$, we study the Jacobian algebras associated with closed oriented surfaces with a non-empty finite collection of punctures, excluding only the case of a sphere with four (or fewer) punctures. The algebras associated with a sphere with five punctures are associated with a quotient of a skew-gentle algebra, and we produce an independent dense pair of pointed modules over the associated gentle algebra (Theorem \ref{NSDP}), which is non-domestic (Proposition \ref{NDSG}), whereas in the rest of the cases these Jacobian algebras are presented as the quotient of non-domestic string algebras (Proposition \ref{NDS}). We end the section with the statement that the Jacobian algebras have a superdecomposable module and their KG-dimension is undefined (Theorem \ref{JSD} and Corollary \ref{JKG}). In Section $(5)$, we explicitly show the existence of superdecomposable modules over some well-known algebras like the incidence algebra of the Nazarova–Zavadskij poset, diamond algebra, a garland and some surface algebra as an application of Theorem \ref{NSIDPM}. Finally, in Section $(6)$, we show that a superdecomposable module over an algebra is preserved when we consider its trivial (or relation) extension algebra (Theorem \ref{TSDM}), but the converse is not true in general. We produce a counter-example (Example \ref{BGE}) of a Brauer graph algebra, as it is obtained as a trivial extension algebra over a gentle algebra. Moreover, we comment on the existence of such a module over non-domestic (skew-) Brauer graph algebra in general.

\subsection*{Acknowledgements:}
The author had a stay at the CEMIM-Universidad Nacional de Mar del Plata, supported by a postdoctoral scholarship from CONICET.

\section{Preliminaries}
In this section, we discuss the lattice of pointed modules and the criteria for the existence of super-decomposable modules over an algebra. We also present some basic information about the skew group algebras and the Galois semi-covering functor over a module category (follow \cite{KaPa14} for details), and their interactions with the pointed modules.

\subsection{The lattice of pointed modules and super-decomposable modules}
Here we present a sufficient condition for the existence of a super-decomposable pure-injective module in terms of independent pairs of dense chains of pointed modules. Let $\theta$ be a finitely presented $\Lambda$-module. A $\theta$-pointed $\Lambda$-module (or, simply a pointed $\Lambda$-module) represents a pair $(M; \chi_ M)$ where $M$ is a finitely presented module and $\chi_M: \theta\to M$ is a module homomorphism. If $\theta= \Lambda$, then $\chi_M$ is uniquely determined by $\chi_M(1)$ of $M$. Here, identify $(M; \chi_M)$ with $(M; \chi_M(1))$ and call it a pointed module. A homomorphism $f: M \to N$ is a pointed homomorphism from $(M; m)$ to $(N; n)$ if and only if $f(m)= n$ and call $f: (M; m) \to (N; n)$ a pointed homomorphism. 

Assume that $(M; \chi_M)$ and $(N; \chi_N)$ are two $\theta$-pointed modules. A $\theta$-pointed homomorphism $f: (M; \chi_M) \to (N; \chi_N)$ stand for a homomorphism $f: M \to N$ that satisfy $f \chi_M = \chi_N$. If $f: M \to N$ is an isomorphism, call $f: (M; \chi_M) \to (N; \chi_N)$ a $\theta$-pointed isomorphism. Let $P_\Lambda^\theta$ be the set of all $\theta$-isomorphism classes of $\theta$-pointed $\Lambda$-modules. Let $\equiv$ be the equivalence relation on $P_\Lambda^\theta$ defined by $(M; \chi_M) \equiv (N; \chi_N)$ if and only if there exist pointed homomorphisms $f:(M; \chi_M) \to (N; \chi_N)$ and $g:(N; \chi_N) \to (M; \chi_M)$. The quotient set $\mathcal{P}_\Lambda^\theta:= P_\Lambda^\theta/\equiv$ is a poset with respect to the relation $\leq$ defined by $\overline{(M; \chi_M)} \leq \overline{(N; \chi_N)}$ if and only if there exists a pointed homomorphism $f: \overline{(N; \chi_N)} \to \overline{(M; \chi_M)}$. Denote by $\overline{(S; \chi_S)}$ the $\equiv$-class of a $\theta$-pointed $\Lambda$-module $(S; \chi_S)$. The poset $\mathcal{P}_\Lambda^\theta$ is a modular lattice [Pr1] with respect to the operations $\oplus$ and $\ast$ defined below.

Given two pointed modules $(M; \chi_M), (N; \chi_N)$, a pointed module $(M \oplus N; \chi_{M\oplus N})$ where $\chi_{M\oplus N}(l)= (\chi_M(l); \chi_N(l))$ for any $l\in \theta$, is defined as their pointed direct sum. Set $(M; \chi_M)\oplus (N; \chi_N)=(M \oplus N; \chi_{M\oplus N})$. Let $M\ast N$ is the pushout of $\chi_M$ and $\chi_N$, i.e, $M \ast N= M \oplus N/ \{(\chi_M(l); -\chi_N(l)): l\in \theta\}$, and $\epsilon_M: M \to M \ast N$ and $\epsilon_N: N \to M \ast N$ be the homomorphisms given by $\epsilon_M (m)= \overline{(m; 0)}$ and $\epsilon_N(n)= \overline{(0; n)}$ for any $m\in M$ and $n\in N$. The pointed module $(M\ast N; \chi_{M\ast N})$ where $\chi_{M\ast N}= \epsilon_M \chi_M = \epsilon_N\chi_N$ is the pointed pushout of $(M; \chi_M)$ and $(N; \chi_N)$. Set $(M; \chi_M) \ast (N; \chi_N)= (M \ast N; \chi_{M\ast N})$. Here
$$\mathrm{sup}\{\overline{(M; \chi_M)}, \overline{(N; \chi_N)}\}=\overline{(M\oplus N; \chi_{M\oplus N})};$$
$$\mathrm{inf}\{\overline{(M; \chi_M)}, \overline{(N; \chi_N)}\}=\overline{(M\ast N; \chi_{M\ast N})}.$$
If $\theta = \Lambda^t$, then $\mathcal{P}_\Lambda^\theta$ is equivalent to the lattice of all pp-formulae with $t$ free variables [Pr1]. We say that a lattice $\mathcal{L}\subseteq \mathcal{P}_\Lambda^\theta$ of $\theta$-pointed $\Lambda$-modules is wide iff for any $\overline{(M_p, \chi_{M_p})} < \overline{(M_q, \chi_{M_q})} \in \mathcal{L}$ there are incomparable elements $\overline{(M, \chi_M)}, \overline{(N, \chi_N)}$
of $\mathcal{L}$ such that
$$\overline{(M_p, \chi_{M_p})} < \overline{(M, \chi_M)}, \overline{(N, \chi_N)} < \overline{(M_q, \chi_{M_q})},$$
$$\overline{(M_p, \chi_{M_p})} < \overline{(M\ast N, \chi_{M\ast N})}< \overline{(M\oplus N, \chi_{M\oplus N})} < \overline{(M_q, \chi_{M_q})},$$

In case $\mathcal{P}_\Lambda^\theta$ contains a wide sublattice $\mathcal{L}$, we say that the width of $\mathcal{P}_\Lambda^\theta$ is undefined. Given a set $C$, a family $\{(M_q, \chi_{M_q}); q\in C\}$ of $\theta$-pointed $\Lambda$-modules is denoted by $(M_q, \chi_{M_q})_{q\in C}$. Let $\mathcal{Q}$ be the set of rational numbers viewed as a poset with respect to the natural ordering $\leq$. Recall that a poset $P$ is a $\mathcal{Q}$-chain if and only if it is a dense chain without end points. Any $\mathcal{Q}$-chain is isomorphic as a poset with the set $\mathcal{Q}$.

\begin{definition}\label{dense}
Given a $\mathcal{Q}$-chain $C$, a dense chain of $\theta$-pointed $\Lambda$-modules is a family $(M_q; \chi_{M_q})_{q\in C}$ of $\theta$-pointed $\Lambda$-modules such that:   
\begin{enumerate}
\item the endomorphism ring $End_\Lambda(M_q)$ is local and $\chi_{M_q}\neq 0$ for any $q\in C$;

\item $q < q' \in C$ ensures the existence of homomorphisms $\mu_{q;q}:(M_q; \chi_{M_q})\to (M_{q'}; \chi_{M_q'})$;

\item $q\neq q' \in C$ implies that $(M_q; \chi_{M_q})$ and $(M_{q'}; \chi_{M_{q'}})$ are not isomorphic.
\end{enumerate}
\end{definition}

\begin{definition}\label{densepair}
An independent pair of dense chains of $\theta$-pointed $\Lambda$ modules is a pair $((M_q; \chi_{M_q})_{q\in C_1}; (N_t; \chi_{N_t})_{t\in C_2})$ of dense chains of $\theta$-pointed
$\Lambda$-modules such that:
\begin{enumerate}
\item the endomorphism ring $End_\Lambda(M_q \ast N_t)$ is local for any $q \in C_1, t\in C_2$;

\item the pointed module $(M_q; \chi_{M_q}) \ast (N_t; \chi_{N_t})$ is not $\theta$-isomorphic to $(M_{q'}; \chi_{M_q'}) \ast (N_t; \chi_{N_t})$ or to $(M_q; \chi_{M_q}) \ast (N_{t'}; \chi_{N_t'})$ for any $q\neq q' \in C_1$ and $t \neq t' \in C_2$.
\end{enumerate}
\end{definition}

Independent pairs of dense chains of pointed modules generate wide lattices as follows:

\begin{theorem}\label{Unw}
Assume that the pair $((M_q; \chi_{M_q})_{q\in C_1}; (N_t; \chi_{N_t})_{t\in C_2})$ is an independent pair of dense chains of $\theta$-pointed $\Lambda$-modules. Then the lattice $$Gen(\overline{(M_q, \chi_{M_q})_{q\in C_1}} \cup \overline{(N_t, \chi_{N_t})_{t\in C_2}}),$$
which is the smallest sublattice of $\mathcal{P}_\Lambda^\theta$ containing sets $(M_q; \chi_{M_q})_{q\in C_1}$ and $(N_t; \chi_{N_t})_{t\in C_2})$, is a wide lattice. Therefore, the width of the lattice $\mathcal{P}_\Lambda^\theta$ is undefined.
\end{theorem}

It is not known whether the existence of a wide sublattice of $\mathcal{P}_\Lambda^\theta$ (or, equivalently, a wide sublattice of the lattice of all pp-formulas over $\Lambda$) implies the existence of an independent pair of dense chains of pointed modules. The assertion $(1)$ of the following theorem is due to Ziegler in the case of independent pairs of $\mathcal{Q}$-chains in $\mathcal{P}_\Lambda^\theta$.

\begin{theorem}\label{Sup}
Let $\theta$ be a finitely presented module in a countable unital ring $\Lambda$.
\begin{enumerate}
\item If the lattice $\mathcal{P}_\Lambda^\theta$ of $\theta$-pointed $\Lambda$-modules has width undefined, then there exists a super-decomposable pure-injective $\Lambda$-module and $\mathrm{KG}(\Lambda)$ is infinite in this case.
\item If there exists an independent pair $(L; K)$ of $\mathcal{Q}$-chains in $\mathcal{P}_\Lambda^\theta$, then there exists a super-decomposable pure-injective $\Lambda$-module.
\end{enumerate}
\end{theorem}


\subsection{Galois semi-covering functor over module categories:}
Here, we study skew-group algebra and present a Galois semi-covering between the module categories of an algebra and skew-group algebra. Moreover, we show that such a functor is exact and faithful, which will preserve the pushout of the pointed modules.

\noindent{\textbf{Skew group algebras:}} Skew group algebras were first studied in view of representation theory in \cite{ReRi85}. For a finite-dimensional algebra $\Lambda$ over a field $K$ and a finite group $G$ acting on $\Lambda$ by automorphisms, the skew group algebra $\bar\Lambda$ shares many representation-theoretic properties with $\Lambda$ often incarnated in properties of functors between $\mathrm{mod}\mbox{-}\Lambda$ and $\mathrm{mod}\mbox{-}\bar\Lambda$. Let $G$ be a finite group acting on an algebra $\Lambda$ by automorphisms. The skew group algebra $\bar\Lambda$ is the algebra defined by:
\begin{itemize}
    \item its underlying vector space is $\Lambda \otimes_ K KG$;
    \item multiplication is given by $(\lambda \otimes g)(\mu \otimes h) = \bar\Lambda(\mu) \otimes gh$ for $\lambda,\mu\in\Lambda$ and $g,h\in G$, extended by linearity and distributivity.
\end{itemize}
Notice that the algebra $\bar\Lambda$ is not basic in general.

In the PhD thesis \cite{Jose83}, J. A. de la Pe\~na observed that the concept of skewness is identical with the Galois covering in the case of a free group action. This motivated us to study skew group algebra construction when the group action is not necessarily free and introduce the notion of the Galois semi-covering functor. Let us first recall the Galois semi-covering $F_\lambda: \mathrm{mod}\mbox{-}\Lambda \to  \mathrm{mod}\mbox{-}\bar\Lambda$ from \cite{GoSaTr25}.
\begin{definition}
Let $A,B$ be linear categories with $G$ a group acting on $A$. A functor $F: A\to B$ is called a $G$-semi covering if for any $X,Y\in Ob(A)$, the following hold:
$$B(FX,FY)\approx \begin{cases}\bigoplus_{g\in G} A(gX,Y)&\mbox{if } G_{X}\neq G;\\\bigoplus_{g\in G} A(X,gY)&\mbox{if }G_{Y}\neq G;\\A^{|G|}(X,Y)&\mbox{ if }G_{XY}= G.\end{cases}$$    
\end{definition}

\begin{assumption}
Let us fix a finite abelian group $G$ of order $n$ acting on $KQ$ such that $G_{i_0}\neq G$ implies $|O_{i_0}|=n$ for each $i_0\in \tilde I$. Recall that the set of all irreducible representations of $G$, denoted $\hat{G}$, forms a group w.r.t. tensor product of representations with $|\hat{G}|=n$.
\end{assumption}

The vertices of $Q_G$, the underlying quiver of the skew group algebra, are given by
$Q_{G_0}=\{(i_0, \rho)\mid i_0 \in \tilde I, \rho \in \hat{G}_{i_0}\}.$ The idempotent of $(KQ)G$ corresponding to the vertex $(i_0, \rho)$ is $e_{i_0\rho} = i_0\otimes e_\rho \mbox{, where } e_\rho = \frac{1}{|G_{i_0}|} \sum_{g\in G_{i_0}} \rho(g)g$ is an idempotent of $KG_{i_0}$. Consider the idempotent of $KQ_G$,
$\bar e = \sum_{i_0\in \tilde I} \bar {e_{i_0}} \mbox{ where } \bar {e_{i_0}} = \sum_{\rho \in \hat{G}_{i_0}}e_{i_0\rho}.$ We recall the semi-covering functor \cite[Proposition~2.12]{GoSaTr25} induced by $F: KQ \to (KQ)G$ given by setting for each $i \in Q_0$, $F(i)=\bar {e_{i_0}} \mbox{ where, } i_0\in O_i\cap \tilde I$ and its push-down functor, $F_\lambda: \mathrm{mod}\mbox{-}\Lambda \to \mathrm{mod}\mbox{-}\bar\Lambda$. 

\begin{definition}\label{dgsc}
Let $F: \Lambda \to \bar\Lambda$ be the Galois semi-covering functor. We define the pushdown functor $F_\lambda: \mathrm{mod}\mbox{-}\Lambda \to \mathrm{mod}\mbox{-}\bar\Lambda$ as follows:

Suppose $M \in \mathrm{mod}\mbox{-}\Lambda$, then $F_\lambda M= \bigoplus_{i_0\in \tilde I} F_\lambda M(\bar {e_{i_0}})$ where, for each $\bar {e_{i_0}}\in \bar\Lambda$, we set
$$F_\lambda M(\bar {e_{i_0}}):= \bigoplus_{F(x)=\bar {e_{i_0}}}M(x).$$

Assume that $\bar{\alpha}\in \Lambda_1G(\bar {e_{i_0}},\bar {e_{j_0}})$. Consider the following cases:
\begin{enumerate}
\item If $G_{i_0}\neq G$ then there is an arrow $\alpha_h: h{i_0}\to {j_0}$ for some $h\in G$ such that $\bar{\alpha}=F(\alpha_h)$. Then the homomorphism $F_\lambda M(\bar{\alpha}): F_\lambda M(e_{i_0}) \to F_\lambda M(\bar {e_{j_0}})$ is defined by: 
$$(\mu_g) \mapsto (\sum_{g\in G} M(g\alpha_h)(\mu_g)).$$

\item If $G_{i_0}=G$ but $G_{j_0}\neq G$ then there is an arrow $\alpha_h: i_0\to hj_0$ for some $h\in G$ such that $\bar{\alpha}=F(\alpha_h)$. The homomorphism $F_\lambda M(\bar{\alpha}): F_\lambda M(\bar {e_{i_0}}) \to F_\lambda M(e_{j_0})$ is defined by:
$$\mu \mapsto (M(g_1\alpha_h)(\mu),\hdots, M(g_n\alpha_h)(\mu)).$$

\item If $G_{i_0j_0}= G$ then there is an arrow $\alpha: i_0\to j_0$ such that $\bar{\alpha}=F(\alpha)$. The homomorphism $F_\lambda M(\bar{\alpha}): F_\lambda M(\bar {e_{i_0}}) \to F_\lambda M(\bar {e_{j_0}})$ is defined by homomorphism:
$$\mu \mapsto M(\alpha)(\mu).$$
\end{enumerate}
\end{definition}


 

\vspace{0.1in}
\textbf{Morphism in $\mathrm{Mod}\mbox{-}\bar\Lambda$:} 
Assume that $f:=(f_{i_0})_{i_0\in Q_0}: M\to N \in \mathrm{mod}\mbox{-}\Lambda$ is a module homomorphism where, $f_{i_0}: M(i_0) \to N(i_0)$. Then $F_\lambda(f):= (\hat{f}_{\bar{e_{i_0}}}): F_\lambda(M) \to F_\lambda(N)$ where, $\hat{f}_{\bar{e_{i_0}}}: F_\lambda(M)(\bar{e_{i_0}}) \to F_\lambda(N)(\bar{e_{i_0}})$ is defined by homomorphisms $f_{i}: M(i) \to N(i)$, for all $i\in \mathcal{O}(i_0)$. If $M\in \mathrm{mod}\mbox{-}\Lambda$, then $F_\lambda(M)\in \mathrm{mod}\mbox{-}\bar\Lambda$. Hence, the functor $F_\lambda$ restricts to a functor $\mathrm{mod}\mbox{-}\Lambda \to \mathrm{mod}\mbox{-}\bar\Lambda$, also denoted by $F_\lambda$. 

\vspace{0.1in}
\textbf{Group action over $\mathrm{Mod}\mbox{-}\Lambda$:} Given a $M\in \mathrm{mod}\mbox{-}\Lambda$ we denote by ${}^gM$ the module $M\circ g^{-1}$ and a module homomorphism $f: M \to N$ we denote by ${}^gf$ the $\Lambda$-module homomorphism ${}^gM \to {}^gN$ such that ${}^g{f_x}:= f_{g^{-1}x}$, for any $x \in Q_0$. This defines an action of $G$ on $\mathrm{Mod}\mbox{-}\Lambda$. Moreover, the map $f\to {}^gf$ defines isomorphism of vector spaces $\Lambda(M,N) \approx \Lambda({}^gM, {}^gN)$. 

\vspace{0.1in}
The next theorem ensures the existence of a Galois semi-covering from $\mathrm{mod}\mbox{-}\Lambda$ to $\mathrm{mod}\mbox{-}\bar\Lambda$.
\begin{theorem}\label{Skmor}\cite[Theorem~3.11]{GoSaTr25}
Assume that a group $G$ acts on an algebra $\Lambda$ and $F: \Lambda \to \bar\Lambda$ is a Galois semi-covering. Then for any $M, N\in \mathrm{mod}\mbox{-}\Lambda$, the functor $F_\lambda: \mathrm{mod}\mbox{-}\Lambda \to \mathrm{mod}\mbox{-}\bar\Lambda$ induces the following isomorphisms of vector spaces:
\begin{equation}\label{Skmoreq}
\mathrm{Hom}_{\bar\Lambda}(F_\lambda M,F_\lambda N)\approx \begin{cases}\bigoplus_{g\in G} \mathrm{Hom}_\Lambda ({}^gM,N)&\mbox{ if } G_{M}\neq G;\\\bigoplus_{g\in G} \mathrm{Hom}_\Lambda(M,{}^gN)&\mbox{ if }G_{N}\neq G;\\\mathrm{Hom}_\Lambda^{|G|}(M,N)&\mbox{ if }G_{MN}= G.\end{cases}
\end{equation}
\end{theorem}

The following remark indicates that the functor $F_\lambda$ is exact and faithful. 
\begin{rmk}\label{faex}
\begin{enumerate} Let us enlist two important properties of $F_\lambda$:
\item[Exactness] The Riedtmann functor $\bar\Lambda\otimes _\Lambda -:\mathrm{mod}\mbox{-}\Lambda \to \mathrm{mod}\mbox{-}\bar\Lambda$ is equivalent to $F_\lambda: \mathrm{mod}\mbox{-}\Lambda \to \mathrm{mod}\mbox{-}\bar\Lambda$ \cite{Sard25}. Since $\bar\Lambda$ is a faithfully flat $\Lambda$-module, $F_\lambda$ is an exact functor.

\item[Faithfulness] Composing the isomorphism in Theorem \ref{Skmor} with the inclusion of ${}_\Lambda(M, N)$ into the morphism space on the right hand side of equation \ref{Skmoreq} in different cases (for example,  ${}_\Lambda(M, N)\hookrightarrow{} \bigoplus {}_{\Lambda}(gM,N)$ if $G_M\neq G$), we obtain a monomorphism of the form $f\mapsto \nu_{M,N}(f)$. This shows that, $F_\lambda$ is also faithful.
\end{enumerate}
\end{rmk}

Since $\Lambda$ is a finite-dimensional algebra, then any finitely presented $\Lambda$-module is finite-dimensional. Thus, we can replace the locality of endomorphism rings with the indecomposability of corresponding modules in Definition \ref{dense} and Definition \ref{densepair}. But the following remark \cite[Propositions~3.6, 3.7]{GoSaTr25} shows that, unlike Galois covering, Galois semi-covering does not necessarily preserve indecomposable modules.
\begin{rmk}\label{indsk}
For all $\bar M \in \mathrm{Ind}\mbox{-}\bar\Lambda$ and ${g\in G}$, there exists a unique $M \in \mathrm{Ind}\mbox{-}\Lambda$, we have
$$F_\lambda (^gM)= \begin{cases}\bar M &\mbox{ if } G_{M}\neq G;\\\bigoplus_{\hat{g}\in \hat{G}} \hat{g} \bar M &\mbox{ if }G_{M}= G.\end{cases}$$ 
\end{rmk}

S. Kasjan and G. Pastuszak show the existence of an independent pair of dense chains under an additive fully faithful right-exact functor or an additive fully faithful exact functor \cite[Theorem 3.13]{KaPa14}. But the following remark says that $F_\lambda$ is not full.   
\begin{rmk}
Suppose $G$ is a finite abelian group of $K$-linear automorphisms of $\Lambda$, and $\bar\Lambda$ is the associated skew group algebra. Then Remark \ref{indsk} ensures that the exact faithful Riedtmann functor $F_\lambda: \mathrm{mod}\mbox{-}\Lambda \to \mathrm{mod}\mbox{-}\bar\Lambda$ is not full.
\end{rmk}

The question is whether $F_\lambda$ preserves independent pairs of dense chains of pointed modules and wide posets of pointed modules (even under certain assumptions).  Let $\theta$ be a finitely presented $\Lambda$-module. Since $F_\lambda$ is exact, $F_\lambda(\theta)$ is also a finitely presented $\bar{\Lambda}$-module. If $(M; \chi_M)$ is a $\theta$-pointed $\Lambda$-module, we set $\chi_{F_\lambda(M)}= F_\lambda(\chi_M)$ and
$$F_\lambda(M; \chi_M):= (F_\lambda(M); F_\lambda(\chi_M))= (F_\lambda(M); \chi_{F_\lambda(M)}).$$

Then, by our assumptions, $F_\lambda(M; \chi_M)$ is an $F_\lambda(\theta)$-pointed $\bar\Lambda$-module. Recall that a $\theta$-pointed module $(M; \chi_M)$ is identified with the pointed module $(M; \chi_M(1))$ if $\theta = \Lambda$. Thus, set $F_\lambda(M; m):= F_\lambda(M; \chi_M)$ where $\chi_M: \Lambda \to M$ is a homomorphism defined by $\chi_M(1) = m$, for any pointed $\Lambda$-module $(M;m)$. This implies that $F_\lambda(M; \chi_M)$ is a well defined $F_\lambda(\theta)$-pointed $\bar\Lambda$-module for any $\theta$-pointed $\Lambda$-module $(M; \chi_M)$.

\begin{proposition}\label{pm}
Assume that $G$ acts on an algebra $\Lambda$ with $\bar{\Lambda}$ as its associated skew group algebra. Then for any $\theta$-pointed $\Lambda$-modules $(M; \chi_M)$ and $(N; \chi_N)$, we have 
\begin{enumerate}
\item $F_\lambda((M; \chi_M) \ast (N; \chi_N))\equiv F_\lambda(M; \chi_M)\ast F_\lambda(N; \chi_N)$;
\item $(M; \chi_M)\leq (N; \chi_N)$ if and only if $F_\lambda(M; \chi_M) \leq F_\lambda(N; \chi_N)$.
\end{enumerate}   
\end{proposition}
\begin{proof}
$(1)$ Since $F_\lambda$ is exact by Remark \ref{faex}, the result follows from \cite[Proposition 3.11]{KaPa14}.

$(2)$ If $(M; \chi_M)\leq (N; \chi_N)$ then there exists a pointed homomorphism $f:(M; \chi_M)\to (N; \chi_N)$ i.e. $f:M\to N$ with $f\chi_M=\chi_N$. Remark \ref{faex} states that $F_\lambda$ is faithful, and thus the assertion holds since faithful functors reflect commutative triangles. 
\end{proof}

\section{Superdecomposable modules over skew gentle algebras}
G. Puninski shows the existence of superdecomposable pure-injective modules over any arbitrary non-domestic string algebra in \cite{Puni04}. In this section, we present some properties that are fundamental for constructing an independent pair of dense chains of pointed modules over a string algebra and show that skewness preserves superdecomposable modules under certain conditions by considering the Galois semi-covering functor between the module categories of a gentle algebra and its associated skew-gentle algebra.

\vspace{0.05in}
\noindent{\textbf{Classification problem:}} We recall that a finite-dimensional algebra $\Lambda$ over an algebraically closed field $K$ is tame if, for any natural number $d$, the set of all isomorphism classes of $d$-dimensional indecomposable $\Lambda$-modules can be parameterised, up to finitely many elements, by a finite number of regular $1$-parameter families. Let $\mu_\Lambda(d)$ be the minimal number of $1$-parameter families needed to parameterise the $d$-dimensional indecomposables in the above sense. A (tame) algebra $\Lambda$ is of polynomial growth (resp. domestic) \cite{CMR80} if there exists a natural number $g$ such that $\mu_\Lambda(d)\leq d^g$ (resp. $\mu_\Lambda(d)\leq g$) for all $d\geq 1$ and $\Lambda$ is $m$-domestic if $m$ is the least such integer $g$. We say that $\Lambda$ is tame of exponential growth if $\mu_\Lambda(d)> r^d$ for infinitely many $d\in \mathbb{N}$ and some real number $r>1$.

\vspace{0.05in}
\noindent{\textbf{Special biserial algebras:}} Given a quiver $Q= (Q_0,Q_1,s,t)$ and an admissible ideal $I$ of $KQ$, the algebra $KQ/I$ is called a special biserial algebra if 
\begin{enumerate}
\item at every vertex $v$ in $Q$ there are at most two arrows starting at $v$ and there are at most two arrows ending at $v$;
\item for every arrow $\alpha$ in $Q$ there exists at most one arrow $\beta$ such that $\alpha \beta\notin I$ and there
exists at most one arrow $\gamma$ such that $\gamma\alpha \notin I$.
\end{enumerate}

If the ideal $I$ is generated by paths of length at least two, then it is known as a string algebra. Gentle algebras form a subclass of string algebras for which the ideal $I$ is generated by paths of length precisely two. In this case, we say $(Q,I)$ is a gentle pair.

By a word of $Q$, we mean a concatenation $\alpha= \alpha_1\hdots \alpha_n$ of symbols of $Q$ such that $t(\alpha_i)= s(\alpha_{i+1})$ for each $1\leq i < n$. The source of $\alpha$ is defined to be $s(\alpha)= s(\alpha_1)$ and the target of $\alpha$ is defined to be $t(\alpha)= t(\alpha_n)$. The inverse of $\alpha$ is defined to be the word $\alpha^{-1}= \alpha_n^{-1}\hdots \alpha_1^{-1}$. A string $u$ is a word $\alpha=  \alpha_1\hdots \alpha_n$ such that $\alpha_i \neq \alpha_{i+1}^{-1}$ for all $1\leq i\leq n-1$ and no sub-string of $u$ or its inverse is in $I$. A band is defined to be a cyclic string $b$ such that each power $b^m$ is a string, but $b$ is not a proper power of any string. If a string algebra has finitely many bands, then it is called domestic; if not, it is non-domestic. For a string $S= \alpha_1\hdots \alpha_n$, the string module associated to S is an $(n+1)$-dimensional module $M(S)$ with $K$-basis $\{z_1,\hdots z_{n+1}\}$ (the canonical basis of $M(S)$) and with multiplication defined as:

$$a· z_j=\begin{cases} z_{j-1} &\mbox{ if } j\geq 2 \mbox{ and } a= \alpha_{j-1};\\z_{j+1} &\mbox{ if } j\leq n \mbox{ and } a^{-1}= \alpha_j; \\0 &\mbox{ otherwise. } \end{cases}$$ for any direct arrow $a$ and $j\in \{1, \hdots, n+1\}$. Any string module is indecomposable and $M(S_1)\approx M(S_2)$ if and only if $S_1= S_2$ or $S_1= S_2^{-1}$. For more details, follow \cite{BR}.

\vspace{0.05in}
\noindent{\textbf{Skew-gentle algebras:}} Skew-gentle algebras belong to the class of clannish algebras introduced by Crawley Boevey \cite{CB89}, which are found as the skew-group algebras of gentle algebras equipped with a certain $\mathbb{Z}_2$-action when the characteristic of $K$ is different from two. Let $(Q,I)$ be a gentle pair. Let $Sp$ be a subset of the vertices of the quiver $Q$ and the elements of $Sp$ are called special vertices; the remaining vertices are ordinary. For a triple $(Q, Sp, I)$, consider the pair $(Q^{Sp}, I^{Sp})$, where $Q_0^{Sp}:= Q_0, Q_1^{Sp}:= Q_1\cup \{\alpha_i\mid i\in Sp, s(\alpha_i)= t(\alpha_i) = i\}$ and $I^{Sp}:= I\cup \{\alpha_i^2\mid i\in Sp\}$. A triple $(Q, Sp, I)$ is called skew-gentle if the corresponding pair $(Q^{Sp}, I^{Sp})$ is gentle. We associate with each vertex $i\in Q_0$ a set, denoted by $Q_0(i)$, as follows: If $i$ is an ordinary vertex then $Q_0(i)= \{i\}$; if $i$ is special then $Q_0(i)= \{i^-, i^+\}$. We denote by $(Q^{sg}, I^{sg})$ the pair defined as follows:
\vspace{-0.01in}
$$Q_0^{sg}:= \cup _{i\in Q_0} Q_0(i),$$
$$Q_1^{sg}[a, b]:= \{(a, \alpha, b)\mid \alpha\in Q_1, a\in Q_0(s(\alpha)), b\in Q_0(t(\alpha))\},$$
$$I^{sg}:= \{\sum_{b\in Q_0(s(\alpha))} \lambda_b(b, \alpha, c)(a, \beta, b)\mid \alpha \beta\in I, a\in Q_0(s(\beta)), c\in Q_0(t(\alpha))\},$$
where $\lambda_b= -1$ if $b= i^-$ for some $i\in Q_0$, and $\lambda_b= 1$ otherwise. Note that the relations in $I^{sg}$ are zero relations or commutative relations. A $K$-algebra $\bar{\Lambda}$ is called skewed-gentle if it is Morita equivalent to a factor algebra $KQ^{sg}/\langle I^{sg}\rangle$, where the triple $(Q, Sp, I)$ is skewed-gentle. The corresponding pair $(Q^{sg}, I^{sg})$ is also said to be skewed-gentle.

Following \cite{GePe99}, we recall below the construction of the gentle algebra $\Lambda$ for any given skewed-gentle triple $(Q, Sp, I)$ such that the corresponding skewed-gentle algebra $\bar{\Lambda}:=KQ^{sg}/\langle I^{sg}\rangle$ is Morita equivalent to a skew-group algebra $\Lambda\mathbb{Z}_2$ in the case of $\mathrm{char}K\neq 2$. 

\vspace{0.05in}
\noindent{\textbf{Group action over a gentle pair:}} For a given special (resp., ordinary) vertex $i$, denote by $Q_0[i]$ the set $\{i\}$ (resp., $\{i^-, i^+\}$). Consider the pair $(Q^g, I^g)$, where $Q_0^g:= \sum_{i\in Q_0}Q_0[i], Q_1^g:= \{\alpha^+, \alpha^-\mid \alpha \in Q_1\}$, and 

\begin{minipage}[b]{0.48\linewidth}
\centering
$$s(\alpha^\pm):= \begin{cases} s(\alpha)^\pm, &\mbox{ if } s(\alpha) \notin Sp;\\s(\alpha), &\mbox{ if } s(\alpha) \in Sp,\end{cases}$$  
    \end{minipage}
\hspace{0.4cm}
\begin{minipage}[b]{0.46\linewidth}
\centering
$$t(\alpha^\pm):= \begin{cases} t(\alpha)^\pm, &\mbox{ if } t(\alpha) \notin Sp;\\t(\alpha), &\mbox{ if } t(\alpha) \in Sp.\end{cases}$$
    \end{minipage}

Moreover, $I^g:= \{\beta^+\alpha^+, \beta^-\alpha^-\mid \beta\alpha\in I, t(\alpha)\notin Sp\}\cup \{\beta^+\alpha^-, \beta^-\alpha^+\mid \beta\alpha\in I, t(\alpha)\in Sp\}$.

The algebra $\Lambda:= KQ^g/\langle I^g\rangle$ is gentle \cite{GePe99}. Call $(Q^g, I^g)$ (resp., $KQ^g/\langle I^g\rangle$) the associated gentle pair (resp., associated gentle algebra) of $(Q, Sp, I)$ or $KQ^{sg}/\langle I^{sg}\rangle$. The group action of $\mathbb{Z}_2=\langle g\rangle$ over $\Lambda$ is defined as: $g(i^\pm):= i^\mp, g(j):= j, g(\alpha^\pm):= \alpha^\mp$ for all $i\in Q_0\setminus Sp, j\in Sp$ and $\alpha\in Q_1$. Then the obtained skew-group algebra $\Lambda\mathbb{Z}_2$ is Morita equivalent to $\bar{\Lambda}$.

\vspace{0.05in}
\noindent{\textbf{Independent dense pair of pointed modules over string algebra:}}
Given an arrow $a\in Q_1$, we define $S(a)$ to be the set of strings over a string algebra $\Lambda$ that start with $a$. We recall from \cite{BR} that there is a linear ordering $<$ on $S(a)$ such that $S < T$ if and only if one of the following conditions is satisfied:
\begin{itemize}
\item $S \phi^{-1}U = T$ for an arrow $\phi$ and a string $U$;
\item $S= T \psi V$ for an arrow $\psi$ and a string $V$;
\item $S= S' \psi W$ and $T= S'\phi^{-1}X$ for some arrows $\phi, \psi$ and strings $S', W, X$.
\end{itemize}

A string algebra $\Lambda$ possesses two different bands $U$ and $V$ starting with the same direct arrow and ending with the same inverse arrow such that $U$ is not a prolongation of $V$ and $V$ is not a prolongation of $U$, i.e. $U\neq VX$ and $V\neq UY$ for any strings $X$ and $Y$. Moreover, assume that $U< V$ and let $\Sigma(U;V)$ be the set of all finite words over the alphabet $\{U;V\}$, including the empty word $\phi$. The next theorem provides a dense chain of pointed modules.

\begin{theorem}\cite[5.3]{KaPa11}\label{DCWE}
Assume that $U$ and $V$ are two different bands such that $U< V$ and $U$ is not a prolongation of $V$ or vice versa. Let $S, T\in \Sigma(U,V)\}$. Then the set $L^T_S(U,V):= \{SXTU; X\in \Sigma(U,V)\}$ is a dense chain without endpoints.    
\end{theorem} 

A pair $(U, V)$ of two different bands over the string algebra $\Lambda$ that start with the same direct arrow and end with the same inverse arrow is $Q$-generating provided $U<V$ and $U$ is not a prolongation of $V$, or vice versa. For such a pair $(U, V)$ of bands over the string algebra $\Lambda$, the above theorem implies that the set $L^T_S(U,V)$ is isomorphic to the poset $Q$ of rational numbers for any strings $S, T\in \Sigma(U,V)$. Given two $Q$-generating pairs $(U, V)$ and $(U^{-1}, V^{-1})$ and $S,T\in \Sigma(U,V); S',T'\in \Sigma(U^{-1}, V^{-1})$ we can produce an independent pair of dense chains of pointed modules. Assume that $S=s_1\hdots s_n$ is a string over the string algebra $\Lambda$, $M(S)$ is the associated string module, and $z_1^S\in M(S)$ is the first element of the canonical $K$-basis of $M(S)$. We call the pointed $\Lambda$-module $(M(S), z_1^S)$ the canonical pointed string module associated with $S$. Although $M(S)\approx M(S^{-1})$, usually the pointed modules $(M(S), z_1^S)$ and $(M(S^{-1}), z_1^{S^{-1}})$ are not isomorphic. The next lemma indicates how the relation between strings determines a homomorphism between the associated modules.

\begin{lemma}\cite[3.1]{Pu08}\label{MBGA}
Assume that $a\in Q_1, S,T\in S(a)$ and $S < T$. Then there is a pointed $\Lambda$-homomorphism $f_{(S,T)}: (M(S); z_1^S)\to (M(T), z_1^T)$ of the canonical pointed string modules $(M(T), z_1^T)$ and $(M(S), z_1^S)$.    
\end{lemma} 

Suppose $T= t_1\hdots t_k, S= s_1\hdots s_m$ are strings over $\Lambda$ such that $TS$ is also a string. Denote by $z^{(T,S)}$ the element $z^{TS}_{k+1}$ of the canonical basis $(z^{TS}_1\hdots z^{TS}_{k+1}\hdots z^{TS}_{k+m+1})$ of $M(TS)$. We discuss the pointed pushout of the pointed modules in the next lemma. 

\begin{lemma}\cite[3.2]{Pu08} 
Assume that $T,S$ are strings over $\Lambda$ such that $T^{-1}S$ is also a string. Then the pointed module $(M(T^{-1}S), z^{(T^{-1},S)})$ is the pointed pushout of the pointed modules $(M(S), z_1^S)$ and $(M(T), z_1^T)$.
\end{lemma} 

We refer to \cite[Section~2]{Sch00} for a combinatorial description of pointed homomorphisms between string modules. 

\begin{theorem}\cite[5.7]{KaPa11}\label{IPDC}
Assume that $(U,V)$ and $(U^{-1}, V^{-1})$ are $Q$-generating pairs over the string algebra $\Lambda$, and let $S, T\in \Sigma(U,V), S', T'\in \Sigma(U^{-1}, V^{-1})$. Then $$((M(X), z_1^X)_{X\in L^T_S(U,V)}, (M(Y), z_1^Y)_{Y\in L^{T'}_{S'}(U^{-1},V^{-1})}$$ is an independent pair of dense chains of pointed modules in $\Lambda$-mod.    
\end{theorem} 
The next theorem follows as a consequence of Theorem \ref{Sup}. 
\begin{theorem}\cite{Puni04}\label{supstr}
Let $\Lambda$ be a non-domestic gentle algebra. Then the width of the lattice of all pp-formulae over $\bar{\Lambda}$ is equal to infinity, and it has a superdecomposable module.    
\end{theorem}

\begin{definition}\label{nonsy}
An independent pair $((M_q, m_q)_{q\in L_1}, (N_t, n_t)_{t\in L_2})$ of dense chains of pointed modules over $\Lambda$ is non-symmetric if for any $q, q_1, q_2\in L_1$ and $t, t_1, t_2\in L_2$ the following hold:
\begin{enumerate}
\item ${}^gM_{q_1}\ncong M_{q_2}$;
\item ${}^gN_{t_1}\ncong N_{t_2}$;
\item ${}^gP_{q_1t_1}\ncong P_{q_2t_2}$ where, $(P_{qt}, p_{qt}):= (M_q, mq)\ast (N_t, n_t)$.
\end{enumerate}
\end{definition}

The next result ensures the existence of an independent pair of dense chains of pointed modules over the skew group algebra when the given algebra has a non-symmetric one.
\begin{theorem}\label{NSIDPM}
Suppose $\mathrm{char}(K)\neq 2$ and $F_\lambda: \mathrm{mod}\mbox{-}\Lambda \to \mathrm{mod}\mbox{-}\bar\Lambda$ is a Galois semi-covering and the module $\theta$ is semisimple. 
\begin{enumerate}
\item If $((M_q, \chi_{M_q})_{q\in L_1}, (N_t, \chi_{N_t})_{t\in L_2})$ is a non-symmetric pair of dense chains of pointed modules over $\Lambda$, then $((F_\lambda(M_q), \chi_{F_\lambda(M_q)})_{q\in L_1}; (F_\lambda (N_t); \chi_{F_\lambda(N_t)})_{t\in L_2})$ is an independent pair of dense chains of pointed modules over $\bar{\Lambda}$.

\item If $\overline{(M_p; \chi_{M_p})}_{p\in L_1}$ is a wide poset of $\theta$-pointed modules in $P_\Lambda^\theta$, then $\overline{(F_\lambda(M_p); \chi_{F_\lambda(M_p))}}_{p\in L_1}$ is a wide poset of $\tilde{\theta}$-pointed modules in $P^{\tilde{\theta}}_{\bar{\Lambda}}$.
\end{enumerate}
\end{theorem}
\begin{proof}
We first show that, $(F_\lambda(M_q), \tilde{m_q})_{q\in L_1}$ is a dense chain of $\tilde{\theta}$-pointed $\bar\Lambda$-modules (similar arguments show that $(F_\lambda (N_t); \tilde{n_t})_{t\in L_2})$ is a dense chain of $\tilde{\theta}$-pointed $\bar\Lambda$-modules as well). Let $q\in L_1$. Clearly, Remark \ref{indsk} says that the module $F_\lambda(M_q)$ is indecomposable by construction. Since $F_\lambda$ is faithful by Remark \ref{faex}, we get $\chi_{F_\lambda(M_q)}\neq 0$ by Theorem \ref{Skmor}.

\vspace{0.05in}
Let $q < q'$ and $\mu_{q,q'}: (M_q, \chi_{M_q})\to (M_{q'}, \chi_{M_{q'}})$ is a $\theta$-pointed homomorphism. We get a $\tilde{\theta}$-pointed homomorphism $F_\lambda(\mu_{q,q'}): (F_\lambda(M_q); \chi_{F_\lambda(M_q)})\to (F_\lambda(M_{q'}); \chi_{F_\lambda(M_{q'})})$ by Proposition \ref{pm}. Furthermore, $\tilde{\theta}$-pointed modules $(F_\lambda(M_q); \chi_{F_\lambda(M_q)})$ and $(F_\lambda(M_{q'}); \chi_{F_\lambda(M_{q'})})$ are not isomorphic for any $q\neq q'$ because $M_q \not\approx {}^gM_{q'}$ and thus we have $F_\lambda(M_q)\not\approx F_\lambda(M_{q'})$ by Remark \ref{indsk}. Hence $(F_\lambda(M_q); \chi_{F_\lambda(M_q)})_{q\in L_1}$ is a dense chain of $\tilde{\theta}$-pointed $\bar\Lambda$-modules.

\vspace{0.05in}
Now, we show that the pair $((F_\lambda(M_q), \tilde{m_q})_{q\in L_1}; (F_\lambda (N_t); \tilde{n_t})_{t\in L_2})$ forms an independent pair. Proposition \ref{pm} states that $F_\lambda((M; \chi_M) \ast (N; \chi_N))\approx F_\lambda(M; \chi_M)\ast F_\lambda(N; \chi_N)$. Moreover, since $((M_q, m_q)_{q\in L_1}, (N_t, n_t)_{t\in L_2})$ is a non-symmetric pair, $F_\lambda((M; \chi_M) \ast (N; \chi_N))$ is indecomposable by Remark \ref{indsk}. Hence, $F_\lambda(M; \chi_M)\ast F_\lambda(N; \chi_N)$ is indecomposable.

\vspace{0.05in}
Assume that $F_\lambda(M_q, \chi_{M_q})\ast F_\lambda(N_t, \chi_{N_t})$ is isomorphic with $F_\lambda(M_q, \chi_{M_q})\ast F_\lambda(N_{t'}, \chi_{N_{t'}})$, for some $q \in L_1$ and $t\neq t' \in L_2$. Then we get $F_\lambda(M_q\ast N_t)\approx F_\lambda(M_q)\ast F_\lambda(N_t)\approx F_\lambda(M_q)\ast F_\lambda(N_{t'})\approx F_\lambda(M_q \ast N_{t'})$, which yields $M_q \ast N_t \approx {}^g(M_q \ast N_{t'})$. This is a contradiction to the fact that $((M_q, m_q)_{q\in L_1}, (N_t, n_t)_{t\in L_2})$ is a non-symmetric pair. Hence $F_\lambda(M_q, \chi_{M_q})\ast F_\lambda(N_t, \chi_{N_t})$ is not isomorphic with $F_\lambda(M_q, \chi_{M_q})\ast F_\lambda(N_{t'}, \chi_{N_{t'}})$. Similarly, it is not isomorphic with $F_\lambda(M_{q'}, \chi_{M_{q'}})\ast F_\lambda(N_t, \chi_{N_t})$ as well, for any $q'\neq q \in L_1$. Hence, the result follows.    
\end{proof}

Let $\alpha$ and $\beta$ be two different arrows ending in the same vertex over a string algebra $\Lambda$. Then $\beta$ is uniquely determined by $\alpha$. Denote by $B(\alpha)$ the set of all bands with first letter $\alpha$. Then the last letter $\beta^{-1}$ is the same for all bands in $B(\alpha)$. If $\Lambda$ is not domestic, by \cite[p.~41]{Sc97} there is a $\alpha$ for which there are two different bands $U, V\in B(\alpha)$, such that $U$ and $V$ contain no substring of the form $\alpha \beta^{-1}$. From the gentle pair for a given skew-gentle algebra, it is clear that if the group action exchange $\alpha$ and $\beta$ then $(U,V)$ and $(U^{-1}, V^{-1})$ are $Q$-generating pairs over $\Lambda$, and hence $((M(X), z_1^X)_{X\in L^T_S(U,V)}, (M(Y), z_1^Y)_{Y\in L^{T'}_{S'}(U^{-1},V^{-1})}$ forms a non-symmetric independent pair of dense chains of pointed modules over $\Lambda$. Finally, using the Galois semi-covering functor, we get the following result by Theorem \ref{NSIDPM}.

\begin{theorem}\label{SDSG}
Suppose $\bar{\Lambda}$ is a skew-gentle algebra and $\Lambda$ is the associated gentle algebra. If $\Lambda$ is non-domestic, then $\bar{\Lambda}$ has a superdecomposable module. Moreover, $KG(\bar\Lambda)$ is infinite.  
\end{theorem}

M. Prest shows in \cite{Pr13} that the lattice of pp formulas is isomorphic to the lattice of pointed modules using the notion of a factorisable system of morphisms in $\mathrm{mod}\mbox{-}\Lambda$, i.e., a collection of pointed finitely presented modules $(M_q, m_q)_{q\in Q}$ indexed by the rationals such
that $(M_q, m_q)> (M_p, m_p)$ iff $q< p$.

\begin{corollary}\cite[4.4]{Pr13}
If there is no factorisable system of pointed modules in $\mathrm{mod}\mbox{-}\Lambda$ then there is no superdecomposable pure-injective $\Lambda$-module.
\end{corollary} 

If $\bar{\Lambda}$ has a super-decomposable module, then the above corollary ensures the existence of a factorisable system of pointed modules in $\mathrm{mod}\mbox{-}\bar{\Lambda}$. Since $\Lambda$ is morita equivalent to the skew-group algebra $\bar{\Lambda}\hat{G}$ where $\hat{G}$ is the dual group of $G$, consider $G_\lambda:\mathrm{mod}\mbox{-}\bar{\Lambda}\to \mathrm{mod}{-}\Lambda$ \cite[Remark~3.4]{GoSaTr25} to get a factorisable system of pointed modules in $\mathrm{mod}\mbox{-}\Lambda$, which ensures the existence of a recursive system \cite[Section~4.1]{GKS} which leads to the existence of a morphism in the stable radical \cite[Lemma~4.1.5]{GKS}, i.e., the nilpotent index of the radical does not exist. Therefore, $\Lambda$ is non-domestic, as for any domestic string algebra, Schr\"{o}er \cite{Sch00} computed that $\mathrm{rad}^{\omega(n+2)}(\Lambda)= 0$, where $n$ is the maximal length of a path in its bridge quiver. 

\begin{theorem}
Assume that $\bar{\Lambda}$ is a skew gentle algebra. 
\begin{enumerate}
\item $\bar{\Lambda}$ is of a non-domestic type if and only if the lattice $P^{\tilde{\theta}}_{\bar{\Lambda}}$ is wide;
\item If the field k is countable, then $\bar{\Lambda}$ is of non-domestic type if and only if there exists a super-decomposable pure-injective $\bar{\Lambda}$-module. 
\end{enumerate}
\end{theorem}

The following remark deals with Prest's conjecture for skew-gentle algebras.
\begin{rmk}
Skew-gentle algebra is also tame, and we check that it possesses super decomposable modules. Thus it also refutes Prest's first conjecture, but the second conjecture – that an algebra $\Lambda$ is of domestic representation type if and only if there is no super-decomposable pure-injective module – remains valid for skew-gentle algebras.
\end{rmk}

\section{Jacobian algebras associated with closed punctured oriented surfaces:}
In this section, we consider the Jacobian algebras associated with closed oriented surfaces with a non-empty finite collection of punctures, excluding only the case of a sphere with four (or fewer) punctures, and show that these algebras are finite-dimensional, tame algebras, where the algebras associated with a sphere with five punctures are associated with a quotient of skew-gentle algebra, whereas in the rest of the cases these are presented as the quotient of non-domestic string algebras (for details see \cite{GLS16}, \cite{La12}, \cite{Y16}). A non-domestic string algebra has a superdecomposable module by Theorem \ref{supstr}. Finally, we explicitly construct a superdecomposable module in the skew gentle algebra case, which concludes that the Jacobian algebras have superdecomposable modules. 

Y. Valdivieso-D\'{i}az shows in \cite{Y16} that the above-mentioned Jacobian algebras are exponential. The following proposition follows from the proof of the main result, where she shows the existence of a three-cycle and a finite-length cycle which commute in the underlying quiver of any triangulation of a sphere, excluding the case of five (or fewer) punctures. 

\begin{proposition}\label{NDS}
For a sphere excluding the case of five (or fewer) punctures, there exists a triangulation $\mathbb{T}$ such that the Jacobian algebra $\mathcal{P}(Q(\mathbb{T}), W(\mathbb{T}))$ is a quotient of a non-domestic string algebra.       
\end{proposition}

Puninski has proved that super-decomposable pure-injective modules exist over any non-polynomial growth string algebra over a countable field. 

\begin{proposition}\label{NDSG}
For a sphere with five punctures there exists a triangulation $\mathbb{T}$ such that the Jacobian algebra $\mathcal{P}(Q(\mathbb{T}), W(\mathbb{T}))$ is a quotient of a skew-gentle algebra.   
\end{proposition}
\begin{proof}
Consider the skewed-gentle triangulation $\mathbb{T}$ of Fig. \ref{fig:1} (see definition of skewed-gentle triangulation in \cite[Section 6.7]{GLS16}) and the Jacobian algebra $A = \mathcal{P}(Q(\mathbb{T}), W(\mathbb{T})$. Let $I$ be the ideal of $K\langle \langle Q(\mathbb{T})\rangle \rangle$ generated by the set $\delta(W')$ of cyclic derivatives of potential $W'= b_3c_3a_2b_1c_1 + a_1b'_2c'_2a_3$. Consider the quotient algebra $\bar\Lambda=A/I$ (Fig. \ref{fig:2}).

\begin{figure}[h]
\begin{minipage}[b]{0.5\linewidth}
\includegraphics[width=0.6\linewidth, height=0.15\textheight]{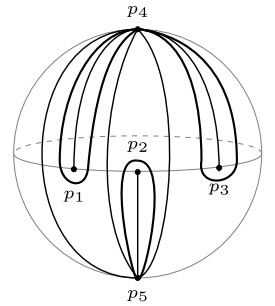}
\caption{Triangulation T of a sphere with 5 punctures.}
\label{fig:1} 
\end{minipage}
\hspace{0.05cm}
\begin{minipage}[b]{0.38\linewidth}
\centering
\begin{tikzcd}[sep={3.0em,between origins}]
                    & 2 \arrow[ld, "c_1"']  &                                                                               & 4 \arrow[rd, "c_2"]    &                                        & 6 \arrow[ld, "c_3"']  &                                         \\
1 \arrow[rr, "a_1"] &                       & 3 \arrow[ru, "b_2"] \arrow[lu, "b_1"'] \arrow[rd, "b'_2"'] \arrow[ld, "b'_1"] &                        & 5 \arrow[ll, "a_2"'] \arrow[rr, "a_3"] &                       & 1 \arrow[lu, "b_3"'] \arrow[ld, "b'_3"] \\
                    & 2' \arrow[lu, "c'_1"] &                                                                               & 4' \arrow[ru, "c'_2"'] &                                        & 6' \arrow[lu, "c'_3"] &                                        
\end{tikzcd}
                                                                                      
    \caption{$Q(\mathbb{T})$}
    \label{fig:2}
    \end{minipage}
\end{figure}

Where, $I=\langle c_3a_2b_1c_1, a_2b_1c_1b_3, b_1c_1b_3c_3,c_1b_3c_3a_2, b_3c_3a_2b_1, b'_2c'_2a_3, c'_2a_3a_1, a_3a_1b'_2, a_1b'_2c'_2 \rangle$. Then $\bar\Lambda$ is isomorphic to a skewed-gentle algebra $K\langle Q'\rangle/J$ where $Q'$ is in Fig. \ref{fig:3} and $J$ is the ideal in $K\langle Q'\rangle$ generated by $\epsilon_i^2-\epsilon_i, a_ib_i, b_ic_i$ and $c_ia_i$ for $i= 1,2,3$ and the set $\{b_2\epsilon_2c_2a_3, \epsilon_2c_2a_3a_1,
c_2a_3a_1b_2, a_1b_2\epsilon_2c_2, \epsilon_3c_3a_2b_1\epsilon_1c_1, c_3a_2b_1\epsilon_1c_1b_3, a_2b_1\epsilon_1c_1b_3\epsilon_3, b_1\epsilon_1c_1b_3\epsilon_3c_3, \\ \epsilon_1c_1b_3\epsilon_3c_3a_2,
c_1b_3\epsilon_3c_3a_2b_1, b_3\epsilon_3c_3a_2b_1\epsilon_1\}$. 

\begin{figure}[h]
\begin{minipage}[b]{0.40\linewidth}
\centering
\begin{tikzcd}[sep={2.8em,between origins}]
                    & 2 \arrow[ld, "c_1"'] \arrow["\epsilon_1", loop, distance=2em, in=125, out=55] &                                        & 4 \arrow[rd, "c_2"] \arrow["\epsilon_2", loop, distance=2em, in=125, out=55] &                                        & 6 \arrow[ld, "c_3"'] \arrow["\epsilon_3", loop, distance=2em, in=125, out=55] &                      \\
1 \arrow[rr, "a_1"] &                                                                               & 3 \arrow[ru, "b_2"] \arrow[lu, "b_1"'] &                                                                              & 5 \arrow[ll, "a_2"'] \arrow[rr, "a_3"] &                                                                               & 1 \arrow[lu, "b_3"']
\end{tikzcd}
    \caption{$Q'$}
    \label{fig:3}
    \end{minipage}
\hspace{0.2cm}
\begin{minipage}[b]{0.50\linewidth}
\centering
\begin{tikzcd}[sep={2.6em,between origins}]
\bar{1} \arrow[rr, "\bar{a}_1"]    &                                                           & \bar{3} \arrow[ld, "\bar{b}_1"'] \arrow[rd, "\bar{b}_2"]    &                                                           & \bar{5} \arrow[ll, "\bar{a}_2"'] \arrow[rr, "\bar{a}_3"]    &                                                           & \bar{1} \arrow[ld, "\bar{b}_3"']  \\
                                   & \bar{2} \arrow[lu, "\bar{c}_1"'] \arrow[ld, "\bar{c}'_1"] &                                                             & \bar{4} \arrow[ru, "\bar{c}_2"] \arrow[rd, "\bar{c}'_2"'] &                                                             & \bar{6} \arrow[lu, "\bar{c}_3"'] \arrow[ld, "\bar{c}'_3"] &                                   \\
\bar{1}' \arrow[rr, "\bar{a}'_1"'] &                                                           & \bar{3}' \arrow[lu, "\bar{b}'_1"] \arrow[ru, "\bar{b}'_2"'] &                                                           & \bar{5}' \arrow[ll, "\bar{a}'_2"] \arrow[rr, "\bar{a}'_3"'] &                                                           & \bar{1}' \arrow[lu, "\bar{b}'_3"]
\end{tikzcd}                                                                                      
    \caption{$Q$}
    \label{fig:4}
    \end{minipage}
\end{figure}

Applying the Gabriel quiver construction \cite{ReRi85}, one can check that the algebra $\bar\Lambda$ is the skew group algebra of the gentle algebra $\Lambda=KQ/I'$ under the action of the group $\mathbb{Z}_2=\langle g \rangle$ defined as $g:\bar{i}\leftrightarrow \bar{i}'$ for $i=1,3,5$ and fixes remaining vertices of the quiver $Q$ in Fig. \ref{fig:4}. Note that the relation here is $I'=\langle \bar{a}_i\bar{b}_i, \bar{b}_i\bar{c}'_i, \bar{c}_i\bar{a}_i, \bar{a}'_i\bar{b}'_i, \bar{b}'_i\bar{c}_i, \bar{c}'_i\bar{a}'_i\rangle$ for $i= 1,2,3$.
\end{proof}

\begin{proposition}\label{QSM}
Let $\Lambda$ be a finite-dimensional algebra and $\Lambda'$ be a quotient of $\Lambda$. If $\Lambda'$ has a superdecomposable module, then the original algebra $\Lambda$ also has one.
\end{proposition}
\begin{proof}
A module $M$ over the quotient algebra $\Lambda'=\Lambda/I$ can be viewed as an $\Lambda$-module (via the natural homomorphism $\Lambda\to \Lambda/I$, where $I$ acts as zero). If $M$ is superdecomposable as an $\Lambda/I$-module, it is also superdecomposable as an $\Lambda$-module because any direct decomposition $M=N\oplus K$ in $\Lambda/I$-modules is also a valid direct decomposition in $\Lambda$-modules.
\end{proof}

Note that the algebra $\Lambda$ in Proposition \ref{NDSG} is a non-domestic string algebra, as the bands $\beta_1:= \bar{a}_3\bar{a}_1{\bar{a}_2}^{-1}$ and $\beta_2:=\bar{a}_1\bar{b}_2\bar{c}_2{\bar{c}_3}^{-1}{\bar{b}_3}^{-1}$ commute with each other. Thus $\Lambda$ has a superdecomposable pure-injective module by Theorem \ref{supstr}. The following theorem ensures the existence of a non-symmetric pair of dense chains of pointed modules over $\Lambda$.  
\begin{theorem}\label{NSDP}
Assume that $(U, V)= (\bar{a}_3\bar{a}_1\bar{a}_2^{-1}, \bar{a}_3\bar{a}_1\bar{b}_2\bar{c}_2\bar{c}_3^{-1}\bar{b}_3^{-1}\bar{a}_1\bar{a}_2^{-1})$. Then $(U,V)$ and $(U^{-1}, V^{-1})$ are $Q$-generating pairs of bands over $\Lambda$. Moreover, the pair
$$((M(X), z_1^X)_{X\in L^T_S(U,V)}, (M(Y), z_1^Y)_{Y\in L^{T'}_{S'}(U^{-1},V^{-1})}$$ is an independent pair of dense chains of pointed modules in $\Lambda$-mod such that:
\begin{enumerate}
\item ${}^gM(S_1)\ncong M(S_2)$ for any $S_1, S_2\in L^U_V(U, V)$,
\item ${}^gM(T_1)\ncong M(T_2)$ for any $T_1, T_2\in L^{U^{-1}}_{V^{-1}}(U^{-1},V^{-1})$,
\item ${}^gM(T_1^{-1}S_1)\ncong M(T_2^{-1}S_2)$ for any $S_1, S_2\in L^U_V(U, V)$ and $T_1, T_2\in L^{U^{-1}}_{V^{-1}}(U^{-1},V^{-1})$.
\end{enumerate}
\end{theorem}
\begin{proof}
$(a)$ We check that $(U, V)$ and $(U^{-1}, V^{-1})$ are $Q$-generating pairs of bands over $\Lambda$ and the pairs are independent pairs of dense chains of pointed modules by Theorem \ref{IPDC}. 

Assume that $S_1= VX_1VU$ and $S_2= VX_2VU$ for some $X_1, X_2\in \Sigma(U, V)$. Thanks to Proposition 5.8, it is enough to show that $S_2\neq {}^gS_1$ and $S_2\neq {}^gS_1^{-1}$.
Obviously ${}^gS_1= {}^g(VX_1VU)= {}^gV{}^gX_1{}^g(VU)$ and ${}^gS_1^{-1}= ({}^gV{}^gX_1{}^g(VU))^{-1}= {}^gU^{-1}{}^gV^{-1}{}^gX_1^{-1}{}^gV^{-1}$. Therefore, $S_2\neq {}^gS_1$ since $V$, as a band, starts with a different arrow than ${}^gV$. Moreover, $S_2\neq {}^gS_1^{-1}$ by Lemma 5.1(a), since $S_2, {}^gS_1^{-1}\in \Sigma(U,V)$ and $S_2$ starts with $V$ and ${}^gS_1^{-1}$ starts with ${}^gU^{-1}$.

$(b)$ Assume that $T_1= V^{-1}Y_1U^{-2}$ and $T_2= V^{-1}Y_2U^{-2}$ for some $Y_1, Y_2\in
\Sigma(U^{-1}, V^{-1})$. We show that $T_2\neq {}^gT_1$ and $T_2\neq {}^gT_1^{-1}$. Here, ${}^gT_1= {}^g(V^{-1}Y_1U^{-2})= {}^gV^{-1}{}^gY_1{}^gU^{-2}$ and ${}^gT_1^{-1}= ({}^gV^{-1}{}^gY_1{}^gU^{-2})^{-1}= {}^gU^2{}^gY_1^{-1}{}^gV$. Thus the assertion follows from the fact that $V$ starts with a different arrow than $V^{-1}$ and $V^{-1}$ starts with a different arrow than ${}^gU$.

$(c)$ Suppose $S_1= VX_1VU, S_2= VX_2VU$ for some $X_1, X_2\in \Sigma(U,V)$ and $T_1=V^{-1}Y_1U^{-2}; T_2=V^{-1}Y_2U^{-2}$ for some $Y_1, Y_2\in \Sigma(U^{-1}, V^{-1})$. We will show that $T_2^{-1}S_2\neq {}^g(T_1^{-1}S_1)$ and $T_2^{-1}S_2\neq {}^g(T_1^{-1}S_1)^{-1}$. Here, $T_2^{-1}S_2= (V^{-1}Y_2U^{-2})^{-1}VX_2VU= U^2Y_2^{-1}V^2X_2VU$; $${}^g(T_1^{-1}S_1)= {}^g(U^2Y_1^{-1}V^2X_1VU)= {}^gU^2{}^gY_1^{-1}{}^gV^2{}^gX_1{}^gV{}^gU;$$ $${}^g(T_1^{-1} S_1)^{-1}={}^gU^{-1}{}^gV^{-1}{}^gX_1^{-1}{}^gV^{-2}{}^gY_1{}^gU^{-2}.$$ So the assertion follows as $U$ starts with a different arrow than ${}^gU$ and ${}^gU^{-1}$.
\end{proof}

Theorems \ref{NSDP} and \ref{SDSG} with Propositions \ref{NDS}, \ref{NDSG}, and \ref{QSM} conclude the next result.  

\begin{theorem}\label{JSD}
Let $S$ be a closed oriented surface with a non-empty finite collection $M$ of punctures, excluding only the case of a sphere with four (or fewer) punctures. For an arbitrary tagged triangulation $T$ in the Jacobian algebra $\Lambda= P(Q(T), W(T))$, there exists an independent pair of dense chains of pointed modules in $\Lambda$-mod. In consequence, there exists a super-decomposable pure-injective $\Lambda$-module, if the base field $K$ is countable. 
\end{theorem}

\begin{corollary}\label{JKG}
Let $S$ be a closed oriented surface with a non-empty finite collection $M$ of punctures, excluding only the case of a sphere with $4$ (or less) punctures. For an arbitrary tagged triangulation $T$, in the Jacobian algebra $\Lambda= P(Q(T), W(T))$, the $\mathrm{KG}(\Lambda)$ is infinite.    
\end{corollary}

\section{Superdecomposable modules over skew algebras and extension algebras}
As an application of Theorem \ref{NSIDPM}, we explicitly show the existence of superdecomposable modules over some well-known skew algebra and their extension algebras.

\vspace{0.1in}
\noindent{\textbf{Superdecomposable modules over incidence algebra of the Nazarova–Zavadskij poset}:}
We show that there exists an independent pair of dense chains of pointed modules over the incidence algebra of the Nazarova–Zavadskij poset, denoted by $\mathcal{NZ}$ in \cite[15.31]{Si92}, which is enlarged by a unique maximal element. More precisely, let $\Gamma:= K(NZ)^\ast$ be the bound quiver algebra $KQ/ \langle\rho \rangle$, where,$(Q, \rho)$ is given in Figure \ref{NZ}. 
\begin{figure}[h]
\begin{minipage}[b]{0.45\linewidth}
\centering
\begin{tikzcd}
v_1 \arrow[r, "\alpha"] \arrow[rdd, "\beta'"', shift right]  & v_2 \arrow[rd, "\gamma"]   &     & v_4 \arrow[ld, "\delta"']   \\
                                                             &                            & v_3 &                             \\
v'_1 \arrow[r, "\alpha'"] \arrow[ruu, "\beta"', shift right] & v'_2 \arrow[ru, "\gamma'"] &     & v_4' \arrow[lu, "\delta'"']
\end{tikzcd}
\caption{$\Gamma$ with $\rho= \{\alpha\gamma+\beta\gamma', \beta'\gamma+\alpha'\gamma'\}$}
\label{NZ}
\end{minipage}
\hspace{0.7cm}
\begin{minipage}[b]{0.40\linewidth}
\centering
\begin{tikzcd}
                                                                                      &                                                              & \bar v_3  &                                                              \\
\bar v_1 \arrow[r, "\bar\beta"', bend right=49] \arrow[r, "\bar\alpha", bend left=49] & \bar v_2 \arrow[ru, "\bar\gamma"] \arrow[rd, "\bar\gamma'"'] &           & \bar v_4 \arrow[lu, "\bar\delta"'] \arrow[ld, "\bar\delta'"] \\
                                                                                      &                                                              & \bar v'_3 &                                                             
\end{tikzcd}
    \caption{$\bar\Gamma$ with $\rho= \{\bar\alpha\bar\gamma', \bar\beta\bar\gamma\}$}
    \label{NZS}
    \end{minipage}
\end{figure}

Applying the Gabriel quiver construction \cite{ReRi85}, one can check that the algebra $\Gamma$ is the skew group algebra of the gentle algebra $\bar\Gamma=K\bar{Q}/\langle \bar{\rho}\rangle$, where,$(\bar{Q}, \bar{\rho})$ is given in Figure \ref{NZS}, under the action of the group $\mathbb{Z}_2=\langle g \rangle$ defined as $g:\bar{\alpha}\leftrightarrow \bar{\beta}$ and $g:\bar{\tau}\leftrightarrow \bar{\tau}'$ for $\tau=\gamma, \delta$.

Consider the bands $U=\bar{\delta}'\bar{\gamma}'^{-1}\bar{\gamma}\bar{\delta}^{-1}$ and $V=\bar{\delta}'\bar{\gamma}'^{-1}\bar{\beta}^{-1}\bar{\alpha}\bar{\gamma}\bar{\delta}^{-1}$. Note that the bands $U$ and $V$ commute. So $\bar{\Gamma}$ is a non-domestic gentle algebra. Since $U< V$ and $U$ is not a prolongation of $V$ or vice versa, we have $(U,V)$ and $(U^{-1}, V^{-1})$ are $Q$-generating pairs over the string algebra $\bar\Gamma$. Then for any $S, T\in \Sigma(U,V), S', T'\in \Sigma(U^{-1}, V^{-1})$, the set $L^T_S(U,V)$ and $L^{T'}_{S'}(U^{-1},V^{-1})$ are dense chains without endpoints by Theorem \ref{DCWE} and $((M(X), z_1^X)_{X\in L^T_S(U,V)}; (M(Y), z_1^Y)_{Y\in L^{T'}_{S'}(U^{-1},V^{-1})})$ is an independent pair of dense chains of pointed $\Lambda$-modules by Theorem \ref{IPDC}. Moreover, since ${}^g(U)=U^{-1}$ and ${}^g(V)=V^{-1}$, using a similar argument in Theorem \ref{NSDP}, one can easily show that this forms a non-symmetric pair. Finally, using Theorem \ref{NSIDPM}, we obtain an independent pair of dense chains of pointed modules and a wide poset over $\Gamma$, thereby showing that $\mathrm{KG}(\Gamma)$ is infinite.

\vspace{0.1in}
\noindent{\textbf{Superdecomposable modules over diamond algebra}:}
We show that there exists an independent pair of dense chains of pointed modules over the diamond algebra $\Lambda_3$. $\bar\Gamma=KQ/\langle \rho\rangle$, where, $(Q, \rho)$ is given in Figure \ref{diamond}. Applying the Gabriel quiver construction \cite{ReRi85}, one can check that the algebra $\Lambda_3$ is the skew group algebra of the algebra $\bar{\Lambda}_5=K\bar{Q}/\langle \bar{\rho}\rangle$, where, $(\bar{Q}, \bar{\rho})$ is given in Figure \ref{tomtom}, under the action of the group $\mathbb{Z}_2=\langle g \rangle$ defined as $g:\bar{v_i}\leftrightarrow \bar{v_i}'$ for $i=1,4$ and fixes the remaining vertices. Moreover, the algebra $\bar{\Lambda}_5$ is the skew group algebra of the gentle algebra $A_1=K\tilde{Q}/\langle \tilde{\rho}\rangle$, where, $(\tilde{Q}, \tilde{\rho})$ is given in Figure \ref{figure8}, under the action of the group $\mathbb{Z}_2=\langle g \rangle$ defined as $g:\tilde{\alpha}\leftrightarrow \tilde{\beta}$ and $g:\tilde{\gamma}\leftrightarrow \tilde{\delta}$. 

\begin{figure}[h]
\begin{minipage}[b]{0.50\linewidth}
\centering
\begin{tikzcd}
                           &                            & v_1 \arrow[ld, "\alpha'"] \arrow[rd, "\beta"'] \arrow[lld, "\alpha"'] \arrow[rrd, "\beta'"] &                           &                             \\
v_2 \arrow[rrd, "\gamma"'] & v'_2 \arrow[rd, "\gamma'"] &                                                                                             & v_3 \arrow[ld, "\delta"'] & v'_3 \arrow[lld, "\delta'"] \\
                           &                            & v_4                                                                                         &                           &                            
\end{tikzcd}
\caption{$\Lambda_3$ with $\bar{\rho}= \{\alpha\gamma+\beta\delta+\beta'\delta',\alpha'\gamma'+\beta\delta+\beta'\delta'\}$}
\label{diamond}
\end{minipage}
\hspace{0.6cm}
\begin{minipage}[b]{0.44\linewidth}
\centering
\begin{tikzcd}
\bar v_1 \arrow[r, "\bar \alpha"] \arrow[rdd, "\bar \beta"', shift right=2]    & \bar v_2 \arrow[r, "\bar \gamma"] \arrow[rdd, "\bar \gamma'", shift left] & \bar v_4  \\
                                                                               &                                                                           &           \\
\bar v'_1 \arrow[r, "\bar \beta'"] \arrow[ruu, "\bar \alpha'"', shift right=2] & \bar v_3 \arrow[r, "\bar \delta'"] \arrow[ruu, "\bar \delta", shift left] & \bar v'_4
\end{tikzcd}
    \caption{$\bar\Lambda_3$ with $\rho= \{\bar\alpha'\bar\gamma+\bar\beta'\bar\delta, \bar\alpha'\bar\gamma'+\bar\beta'\bar\delta',\bar\alpha\bar\gamma'+\bar\beta\bar \delta',\bar\alpha\bar\gamma+\bar\beta\bar\delta\}$}
    \label{tomtom}
    \end{minipage}
\end{figure}

Consider the bands $U=\tilde{\alpha}\tilde{\gamma}\tilde{\delta}^{-1}\tilde{\beta}^{-1}$ and $V=\tilde{\alpha}\tilde{\beta}^{-1}$ over $A_1$. Note that the bands $U$ and $V$ commute. So $A_1$ is a non-domestic gentle algebra. Since $U< V$ and $U$ is not a prolongation of $V$ or vice versa, we have $(U,V)$ and $(U^{-1}, V^{-1})$ are $Q$-generating pairs over the string algebra $\bar\Gamma$. Then for any $S, T\in \Sigma(U,V), S', T'\in \Sigma(U^{-1}, V^{-1})$, the set $L^T_S(U,V)$ and $L^{T'}_{S'}(U^{-1},V^{-1})$ are dense chains without endpoints by Theorem \ref{DCWE} and $((M(X), z_1^X)_{X\in L^T_S(U,V)}; (M(Y), z_1^Y)_{Y\in L^{T'}_{S'}(U^{-1},V^{-1})})$ is an independent pair of dense chains of pointed $A_1$-modules by Theorem \ref{IPDC}. Moreover, since ${}^gU=U^{-1}$ and ${}^gV=V^{-1}$, this pair is non-symmetric, and we obtain an independent pair of dense chains of pointed $\bar{\Lambda}_5$-modules by Theorem \ref{NSIDPM}. Such a pair is $((M(\bar{X}),\bar{z}_1^{\bar{X}})_{\bar{X}\in L^{\bar{T}}_{\bar{S}}(\bar{U},\bar{V})}; (M(\bar{Y}), \bar{z}_1^{\bar{Y}})_{\bar{Y}\in L^{\bar{T}'}_{\bar{S}'}(\bar{U}^{-1},\bar{V}^{-1})})$ where, $\bar{U}=\bar{\alpha}\bar{\gamma}\bar{\delta}^{-1}\bar{\delta}'\bar{\gamma}'^{-1}\bar{\alpha}'^{-1}\bar{\beta}'\bar{\beta}^{-1}$ and $\bar{V}=\bar{\alpha}\bar{\alpha}'^{-1}\bar{\beta}'\bar{\beta}^{-1}$. Here, ${}^g\bar{U}\neq\bar{U}^{-1}\neq \bar{U}$ and ${}^g\bar{V}\neq\bar{V}^{-1}\neq \bar{V}$. Finally, we get an independent pair of dense chains of pointed $\Lambda_3$-modules by Theorem \ref{NSIDPM}.

\vspace{0.1in}
\noindent{\textbf{Superdecomposable modules over a garland}:}
Consider the algebra $A= KQ/I$ where the quiver $Q$ is given in figure \ref{gar} and all squares in the quiver are anti-commutative; it is a garland \cite[Example~3.5]{AB22}. The algebra $A$ is obtained by one-point extension of $\bar{\Lambda}_5$ at $\bar{v}_0$ and one-point coextension at $\bar{v}_5$. We already checked that $\bar{\Lambda}_5$ has a superdecomposable module.

\begin{figure}[h]
\begin{minipage}[b]{0.43\linewidth}
\centering
\begin{tikzcd}
                                                           & \bar v_1 \arrow[r, "\bar \alpha"] \arrow[rdd, "\bar \beta"', shift right=2]    & \bar v_2 \arrow[r, "\bar \gamma"] \arrow[rdd, "\bar \gamma'", shift left] & \bar v_4 \arrow[rd, "\bar \eta"]    &          \\
\bar v_0 \arrow[ru, "\bar \tau"] \arrow[rd, "\bar \tau'"'] &                                                                                &                                                                           &                                     & \bar v_5 \\
                                                           & \bar v'_1 \arrow[r, "\bar \beta'"] \arrow[ruu, "\bar \alpha'"', shift right=2] & \bar v_3 \arrow[r, "\bar \delta'"] \arrow[ruu, "\bar \delta", shift left] & \bar v'_4 \arrow[ru, "\bar \eta'"'] &         
\end{tikzcd}
    \caption{Garland}
    \label{gar}
\end{minipage}
\hspace{0.6cm}
\begin{minipage}[b]{0.5\linewidth}
\centering
\begin{tikzcd}
V \arrow[dd, "g"'] \arrow[rr, "\phi"] &  & {\mathrm{Hom}_\Lambda(M,X)} \arrow[dd, "f_\ast"] \\
                                      &  &                                                  \\
V' \arrow[rr, "\phi'"]                &  & {\mathrm{Hom}_\Lambda(M,X')}                    
\end{tikzcd}
    \caption{Morphism in $\Lambda[M]$}
    \label{tomtom}
    \end{minipage}
\end{figure}

If an algebra $\Lambda$ has a superdecomposable module then its one-point extension algebra $\Lambda[M]=\left(\begin{smallmatrix} K & 0 \\ M & \Lambda \end{smallmatrix}\right))$ where $K$ is a field, also has a superdecomposable module, because the category $\mathrm{mod}\mbox{-}\Lambda[M]$ can be described as follows: the objects are of form a triple $(V, X; \phi)$, where $V$ is a $K$-vector space, $X$ is an $\Lambda$-module, $\phi\in \mathrm{Hom}_K(V, \mathrm{Hom}_\Lambda(M,X))$; the morphisms are of form $(f,g):(V,X;\phi)\to (V',X';\phi')$, where $f \in \mathrm{Hom}_\Lambda(X,X'), g\in \mathrm{Hom}_K(V,V')$, such that the following diagram is commutative i.e, $f_\ast \phi= \phi'g$ where, $f_\ast:=\mathrm{Hom}_\Lambda(\_,f)$; and the compositions of morphisms are natural. The canonical projection $p: \Lambda[M]\to \Lambda$ induces an embedding $\iota_\Lambda: \mathrm{mod}\mbox{-}\Lambda \to \mathrm{mod}\mbox{-}\Lambda[M]$. So we can view $\mathrm{mod}\mbox{-}\Lambda$ as a full subcategory of $\mathrm{mod}\mbox{-}\Lambda[M]$ \cite{PeTr99}. That is, for any $X$ in $\mathrm{mod}\mbox{-}\Lambda$, we write $\iota_\Lambda(X)$ as $X$ inside $\mathrm{mod}\mbox{-}\Lambda[M]$. In particular, if $\Lambda$ has a superdecomposable module $N$, we can often construct a superdecomposable module over $\Lambda[M]$ by using $\iota_\Lambda(N)$ as the $\Lambda$-component $(X)$ of the triple, potentially with $V=0$ chosen such that the structure allows for infinite, non-splitting decompositions.

\vspace{0.1in}
\noindent{\textbf{Superdecomposable modules over a surface algebra}:}
The authors in \cite{AB22} deal with the situation when the associated skew gentle algebras for two different gentle algebras $\Lambda$ and $\Lambda'$ have the same derived category using the geometry of the marked surface. A dissection $D$ on a marked surface $(S, M, P)$ where $M\subseteq \partial S$, is a maximal collection of non-intersecting arcs with endpoints in $M$ or $P$, that do not cut out a subsurface of $S$. To $(S, M, P, D)$, one can associate a gentle algebra. Let $\sigma\in \mathrm{Homeo}^+(S)$ of order $2$ with finitely many fixed points $X$ such that $\sigma(M)= M, \sigma(P)=P$ and $\sigma(D)= D$. This defines a $\mathbb{Z}_2$-action on $\Lambda$.

\begin{figure}[h]
\begin{minipage}[b]{0.32\linewidth}
\centering
\includegraphics[width=0.52\linewidth, height=0.08\textheight]{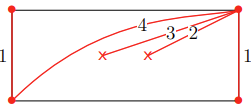}
\caption{$(S, \sigma, D)$}
    \label{orb}
    \end{minipage}
\hspace{0.15cm}
\begin{minipage}[b]{0.3\linewidth}
\centering
\includegraphics[width=0.6\linewidth, height=0.08\textheight]{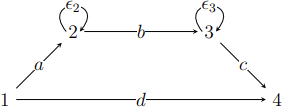}
\caption{$(\bar{Q}, \bar{I}, Sp)$}
    \label{NAD}
    \end{minipage}
\hspace{0.15cm}
\begin{minipage}[b]{0.32\linewidth}
\centering
\includegraphics[width=0.58\linewidth, height=0.08\textheight]{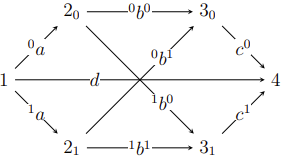}
\caption{$(\tilde{Q}, \tilde{I})$}
    \label{AD}
    \end{minipage}
\end{figure}

Consider the $X$-dissected orbifold corresponding to $(S, \sigma, D)$ is a cylinder with two orbifold points given in Figure \ref{orb}. The corresponding skew-gentle algebra $\bar{\Lambda}$ is given by the following skew-gentle triple $(\bar{Q}, \bar{I}, Sp)$ where $\bar{I}= \{ab, bc\}$ and $Sp= \{\epsilon_2, \epsilon_3\}$ in Figure \ref{NAD}. The associated admissible representation is given by $(\tilde{Q}, \tilde{I})$ where, $\tilde{I}= \{({}^0a)({}^0b^0)+({}^1a)({}^0b^1),({}^0a)({}^1b^0)+({}^1a)({}^1b^1),({}^0b^0)(c^0)+({}^1b^0)(c^1),({}^0b^1)(c^0)+({}^1b^1)(c^1)\}$ in Figure \ref{AD}. Note that the skew-gentle algebra $\bar{\lambda}$ is associated with two different gentle algebras $\Lambda_1=KQ_1/I_1$ and $\Lambda_2=KQ_2/I_2$ with the marked surface $(S_1, \sigma_1)$ (Figure \ref{Surf1}) and $(S_2, \sigma_2)$ (Figure \ref{Surf2}) respectively, where the quivers $Q_1, Q_2$ are given in Figures \ref{Gent1} and \ref{Gent2}, respectively, $I= \{a^+b^+, a^-b^-, b^+c^+, b^-c^-\}$ and $I_2$ are presented by dotted lines.

\begin{figure}[h]
\begin{minipage}[b]{0.48\linewidth}
\centering
\includegraphics[width=0.43\linewidth, height=0.10\textheight]{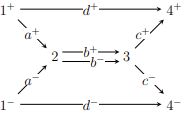}
\caption{$(Q_1,I_1))$}
    \label{Gent1}
    \end{minipage}
\hspace{0.4cm}
\begin{minipage}[b]{0.46\linewidth}
\centering
\includegraphics[width=0.48\linewidth, height=0.11\textheight]{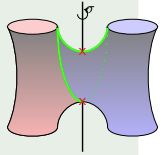}
\caption{$(S_1, \sigma_1)$}
    \label{Surf1}
    \end{minipage}
\end{figure}

Consider the bands $U=d^+{c^+}^{-1}c^-{d^-}^{-1}a^-{a^+}^{-1}$ and $V=d^+{c^+}^{-1}{b^-}^{-1}{a^+}^{-1}$ over $\Lambda_1$. Note that the bands $U$ and $V$ commute. So $\Lambda_1$ is a non-domestic gentle algebra. Since $U< V$ and $U$ is not a prolongation of $V$ or vice versa, we have $(U,V)$ and $(U^{-1}, V^{-1})$ are $Q$-generating pairs. Then the set $((M(X), z_1^X)_{X\in L^T_S(U,V)}; (M(Y), z_1^Y)_{Y\in L^{T'}_{S'}(U^{-1},V^{-1})})$ for any $S, T\in \Sigma(U,V), S', T'\in \Sigma(U^{-1}, V^{-1})$, is an independent pair of dense chains of pointed $\Lambda_1$-modules by Theorem \ref{IPDC}. Moreover, ${}^gU\neq U^{-1}$ and ${}^gV\neq V^{-1}$, this pair is non-symmetric, and we obtain an independent pair of dense chains of pointed $\bar{\Lambda}_5$-modules by Theorem \ref{NSIDPM}.

\begin{figure}[h]
\begin{minipage}[b]{0.5\linewidth}
\centering
\includegraphics[width=0.4\linewidth, height=0.10\textheight]{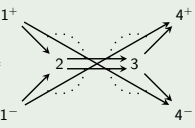}
\caption{$(Q_2,I_2)$}
    \label{Gent2}
    \end{minipage}
\hspace{0.4cm}
\begin{minipage}[b]{0.45\linewidth}
\centering
\includegraphics[width=0.36\linewidth, height=0.10\textheight]{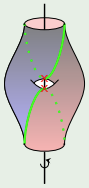}
\caption{$(S_2, \sigma_2)$}
    \label{Surf2}
    \end{minipage}
\end{figure}



\section{Superdecomposable modules over Brauer graph algebras}
Brauer graph algebras originate in the modular representation theory of finite groups, where they first appear in the form of Brauer tree algebras \cite{Ger69}. The author in \cite[Section~ 4.4]{Schrl18} compares derived equivalences of Brauer graph algebras with derived equivalences of frozen Jacobian algebras, with an interpretation of the derived equivalences of Brauer graph algebras as flips of diagonals in triangulations of marked oriented surfaces. In this section, we show the existence of superdecomposable modules over Brauer graph algebras. 

Given an undirected graph $\Gamma$, denote by $val(v)$ the valency of a vertex $v$, i.e. the number of edges incident to $v$ in $\Gamma$. We recall that a Brauer graph $\Gamma$ is a tuple $(\Gamma_0, \Gamma_1, m, o)$ where $(\Gamma_0, \Gamma_1)$ is an unoriented graph with $\Gamma_0$ the set of vertices and $\Gamma_1$ the set of edges and
\begin{itemize}
\item $m$ is a multiplicity function $m: \Gamma_0\to \mathbb{Z}^+$, assigning a positive number to each $v\in \Gamma_0$.
\item $o$ is an orientation of the edges attached to each $v\in \Gamma_0$. For a vertex $v$ such that $val(v) = t$ we denote $o(v)= (x_1, \hdots, x_t)$. In the particular case that $val(v)= 1$ and the only edge attached to $v$ is $x$, if $m(v)= 1$ the cyclic ordering is given by the single element $x$, but when $m(v) > 1$ we set $o(v)= (x, x)$ so that in the cyclic order the successor (and the predecessor) of $x$ is $x$.
\end{itemize}

Following \cite{Schrl18}, given a Brauer graph $\Gamma$, define the Brauer quiver $Q_\Gamma= (Q_0, Q_1)$ as follows:
\begin{enumerate}
\item the set of vertices $Q_0$ is in bijection with the edges of $\Gamma$;
\item there is an arrow $x\to y$ each time $x$ and $y$ are edges incident to a vertex $v\in \Gamma_0$ and $y$ is the immediate successor of $x$ in the cyclic order $o(v)$;
\item each non-distinguished vertex $v$ of $\Gamma$ with $m(v)val(v) \geq 2$ gives rise an oriented cycle $C_v: x_1\xrightarrow{\alpha_1}x_2\to \hdots x_t\xrightarrow{\alpha_t}x_1$ in $Q$, up to cycle equivalence, established by the orientation $o$, i.e. if $o(v)= (x_1,\hdots, x_t)$ (that we will set as clockwise in our examples when $t> 1$). We say that $C_v$ is a special cycle at $v$.
\end{enumerate}

Define an ideal of relations $I_\Gamma$ in $KQ_\Gamma$ generated by three types of relations. For this, we identify the set of edges $\Gamma_1$ of a Brauer graph $\Gamma$ with the set of vertices $Q_0$ of the corresponding quiver $Q_\Gamma$ and that we denote the set of vertices of the Brauer graph by $\Gamma$.

Type I: $C_v^{m(v)}-C_{v'}^{m(v')}$, for any $i\in Q_0$ and for any special $i$-cycles $C_v$ and $C_{v'}$ at $v$ and $v'$ such that both $v$ and $v'$ are not truncated (i.e. $val(v)m(v)\neq 1$ and $val(v')m(v')\neq 1$).

Type II: $C_v^{m(v)}\alpha_1$, for any $i\in Q_0$, $v\in \Gamma_0$ and any special $i$-cycle $C_v= \alpha_1\alpha_2\hdots\alpha_n$.

Type III: $\alpha \beta$, for any $\alpha,\beta\in Q_1$ such that $\alpha\beta$ is not a subpath of any special cycle except if $\alpha= \beta$ is a loop associated to a vertex $v$ of valency one and multiplicity $m(v) > 1$.

Then $\Lambda_\Gamma= KQ_\Gamma/I_\Gamma$ is called the Brauer graph algebra associated to the Brauer graph $\Gamma$.

Brauer graph algebras are special biserial \cite[Theorem 2.8]{Schrl18}, which are obtained as the trivial extension algebra of gentle algebras and the associated Brauer graph is given as the ribbon graph (see \cite[Section~3.1]{Schrl18} for more details) of the gentle algebra. Moreover, the following example \ref{BGE} shows that the associated gentle algebra is not necessarily unique. 

\begin{theorem}\cite[Theorem~1.2]{Sch15} 
Let $\Lambda= KQ/I$ be a gentle algebra with ribbon graph $\Gamma_\Lambda$. Then the trivial extension algebra $T(\Lambda)$ is the Brauer graph algebra with Brauer graph $\Gamma_A$ (and with multiplicity function identically equal to one).
\end{theorem}

Skew Brauer graph algebras are introduced as a generalisation of Brauer graph algebras in \cite[Section~3.1]{So24}, which are found to be the skew-group algebras of Brauer graph algebras equipped with a certain $\mathbb{Z}_2$-action \cite[Proposition~3.9]{So24}. Moreover, (Skew-) Brauer graph algebras (of multiplicity 1) appear as the trivial extension algebra of (skew-) gentle algebras (resp. \cite[Theorem~3.11]{So24} and \cite[Corollary~3.14]{Schrl18}). So if the associated gentle algebra is non-domestic, then the skew-gentle algebra has a superdecomposable module by Theorem \ref{SDSG} and thus the associated (skew-) Brauer graph algebra also has one by Theorem \ref{TSDM}.

\vspace{0.05in}
\noindent{\textbf{Superdecomposable modules over trivial extensions and repetitive algebras}:}
Trivial extensions of Artin algebras play an important role in representation theory. Every finite-dimensional $K$-algebra $\Lambda$ is a quotient of a symmetric $K$-algebra, known as its trivial extension $T(\Lambda)$, where $T(\Lambda)$ is given by the semidirect product of $\Lambda$ with its minimal injective co-generator $D(\Lambda)= \mathrm{Hom}_K(\Lambda, K)$, i.e. $T(\Lambda)= \Lambda \ltimes D(\Lambda)$ with the underlying vector space $\Lambda \bigoplus D(\Lambda)$ and with the multiplication defined by: $$(\lambda_1, f)(\lambda_2, g) = (\lambda_1\lambda_2, \lambda_1g+ f\lambda_2),$$ for any $\lambda_1, \lambda_2\in \Lambda$ and $f, g\in D\Lambda$. We follow the articles (\cite{FST22}, \cite{Sch15}) for the detailed construction of the underlying quiver and relations of $T(\Lambda)$ used in the next section.

A module $M$ over an algebra $\Lambda$ is indeed a module over its trivial extension algebra $T(\Lambda)$ by defining the action of the nilpotent ideal $D(\Lambda)$ to be zero. Specifically, for $(\lambda, f)\in T(\Lambda)$ $m\in M$, the action is defined as $(\lambda, f).m= \lambda m$, which makes $M$ a $T(\Lambda)$-module. 

We show that the Krull-Gabriel dimension of a trivial extension algebra $T(\Lambda)$ exceeds that of its original algebra $\Lambda$ in \cite{Sard25}. Moreover, any direct summand of a $\Lambda$-module $M$ also appears in its decomposition as a $T(\Lambda)$-module, which indeed establishes the existence of a superdecomposable module in $T(\Lambda)$ if $\Lambda$ has one. 

\begin{theorem}\label{TSDM}
Suppose $\Lambda= \mathcal{K}Q_\Lambda/I_\Lambda$ be a finite-dimensional algebra and let $T(\Lambda)= \mathcal{K}Q_{T(\Lambda)}/I_{T(\Lambda)}$ be its trivial extension algebra. Then there exists a superdecomposable module in $T(\Lambda)$ if there exists a superdecomposable module in $\Lambda$.
\end{theorem}

S. Kasjan and G. Pastuszak show the existence of an independent pair of dense chains under an additive fully faithful right-exact functor or an additive fully faithful exact functor \cite[Theorem 3.13]{KaPa14}. Since a Galois covering between two linear categories is a fully faithful and exact functor, particularly when considering the associated push-down functor on module categories, it preserves the superdecomposable modules. Hence, the next remark. 
\begin{rmk}
The repetitive algebra $\hat{\Lambda}$ \cite{Schr99} of a finite-dimensional algebra $\Lambda$ is a Galois covering of its trivial extension algebra $T(\Lambda)$ induced by the Nakayama automorphism, with the Galois group acting as $\mathbb{Z}$. Since a Galois covering preserves a superdecomposable module, the above theorem indicates that $\hat{\Lambda}$ has a superdecomposable module if $\Lambda$ has one. 
\end{rmk}

But the converse is not true in general, i.e., the existence of a superdecomposable module in $T(\Lambda)$ does not imply that $\Lambda$ has one. Below, we consider a few results over gentle algebras $\Lambda$ where $T(\Lambda)$ has a superdecomposable module but $\Lambda$ has none. 
\begin{theorem}\cite{RiMi97}
Assume that $\Lambda$ is a finite-dimensional gentle algebra and $Q(\Lambda)$ has at least two cycles. Then $\hat{\Lambda}$ is non-domestic (thus even of non-polynomial growth).    
\end{theorem} 

Moreover, $\Lambda$ is gentle if and only if $\hat{\Lambda}$ is special biserial \cite{Schr99}.

So if we consider a domestic gentle algebra with at least two cycles, Theorem \ref{supstr} ensures that the associated non-domestic special biserial algebras have a superdecomposable module and the original algebra does not. We produce such an algebra in Example \ref{BGE} below. 

\vspace{0.05in}
\noindent{\textbf{Superdecomposable modules over Brauer graph algebras}:}
First, recall that $\Lambda_\Gamma$ is a special biserial algebra. Here is the comparison of the representation theory of special biserial algebras and string algebras from \cite[Section~5]{Sch00}. Let $\Lambda= KQ/I$ be a special biserial algebra with at least one band and $\hat{I}$ be the set of all the paths $p$ for which $(p-\lambda q) \in I$ for some path $q$ sharing the source and target with $p$ and $\lambda\in K$. In particular, $\hat{I}$ contains all the monomial relations in $I$. Thus the algebra $\hat{\Lambda}:= KQ/\hat{I}$ is a string algebra. Moreover, $I\subseteq \hat{I}$. The surjective map $\Lambda\to \hat{\Lambda}$ induces a canonical $\Lambda$-module structure on every $\hat{\Lambda}$-module. A finite-dimensional indecomposable $\Lambda$-module is either a string module, band module or a projective-injective module. In fact, the only indecomposable $\Lambda$-modules not annihilated by $\hat{I}$ are the projective-injective modules, and there is only a finite number of such modules. Each such module can be written as $P(p-\lambda q)$ for some $(p-\lambda q)\in I$ with $\lambda\in K^*$ and $p,q\notin I$, and occurs in an Auslander-Reiten sequence $0 \to M(C)\to M(C_1)\bigoplus P(p-\lambda q)\bigoplus M(C_2)\to M(D)\to 0$ where $M(C), M(C_1), M(C_2), M(D)$ are string modules over $\hat{\Lambda}$. 

The Brauer graph associated with the domestic Brauer graph algebras is described in \cite[Theorem 5.2]{Schrl18}. Brauer graph algebras are symmetric special biserial algebras, which form a subclass of self-injective special biserial algebras. Bocian and Skowro\'{n}ski give a characterisation of the (non-) domestic Brauer graph algebras.  

\begin{theorem}\cite{ErSk92}
Let $\Lambda\approx KQ/I$ be a self-injective special biserial algebra. Then $\Lambda$ is not of polynomial growth iff $(Q, I)$ has infinitely many bands.    
\end{theorem}

\begin{rmk}\label{SDBGA}
The relation between strings and morphisms between the string modules over Brauer graph algebras is defined in a similar way as described in Lemma \ref{MBGA}, which ensures the existence of superdecomposable modules over non-domestic Brauer graph algebras as in Theorem \ref{IPDC} and their associated skew-Brauer graph algebras as in Theorem \ref{NSIDPM}.     
\end{rmk}

\begin{example}\label{BGE}
Consider the Brauer graph $\Gamma= (\Gamma_0, \Gamma_1, m, o)$ in Figure \ref{BG}. Here, we set $m(a)= m(b)= 1$. The cyclic ordering at vertex $a$ is given by $1< 2 < 3 < 1$ and at vertex $b$ by $1< 2 < 3 < 1$ and $val(a)= val(b) = 3$. 

\begin{figure}[h]
\begin{minipage}[b]{0.47\linewidth}
\centering
\includegraphics[width=0.4\linewidth, height=0.10\textheight]{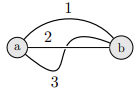}
\caption{Brauer graph $\Gamma$}
    \label{BG}
    \end{minipage}
\hspace{0.4cm}
\begin{minipage}[b]{0.48\linewidth}
\centering
\includegraphics[width=0.48\linewidth, height=0.13\textheight]{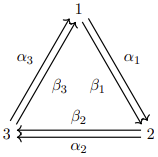}
\caption{Brauer algebra $\Lambda_\Gamma$}
    \label{BGAG}
    \end{minipage}
\end{figure}

The associated Brauer graph algebra is $\Lambda_\Gamma= KQ_\Gamma/I_\Gamma$, where $Q_\Gamma$ is given in Figure \ref{BGAG}. Here, the special $1$-cycle at $a$ is given by $\alpha_1\alpha_2\alpha_3$, the special $1$-cycle at $b$ is given by $\beta_1\beta_2\beta_3$, etc. Moreover, the set of relations is given as follows:

Type I: $\alpha_1\alpha_2\alpha_3-\beta_1\beta_2\beta_3, \alpha_2\alpha_3\alpha_1-\beta_2\beta_3\beta_1, \alpha_3\alpha_1\alpha_2 -\beta_3\beta_1\beta_2$;

Type II: $\alpha^4, \beta^4$;

Type III: $\alpha_i\beta_{i+1}, \beta_i\alpha_{i+1}$ for $i= 1, 2$ and $\alpha_3\beta_1, \beta_3\alpha_1$.

A minimal set of relations is given by all relations of types I and III.  

\begin{figure}[h]
\begin{minipage}[b]{0.40\linewidth}
\centering
\begin{tikzcd}
\tilde{v}_1 \arrow[r, "\tilde{\alpha}", bend left=49] \arrow[r, "\tilde{\beta}"', bend right=49] & \tilde{v}_2 \arrow[r, "\tilde{\gamma}", bend left=49] \arrow[r, "\tilde{\delta}"', bend right=49] & \tilde{v}_3
\end{tikzcd} 
\caption{$A_1$ with $\tilde{\rho}= \{\tilde{\alpha}\tilde{\delta}, \tilde{\beta}\tilde{\gamma}\}$}
    \label{figure8}
    \end{minipage}
\hspace{0.4cm}
\begin{minipage}[b]{0.55\linewidth}
\centering
\includegraphics[width=0.48\linewidth, height=0.13\textheight]{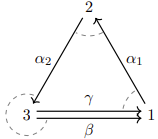}
\caption{$A_2$ with $\rho= \{\alpha_1\alpha_2, \gamma\alpha_1, \alpha_2\beta\}$}
    \label{BGtriangle}
    \end{minipage}
\end{figure}
The algebras $A_1= KQ_1/\langle \tilde{\alpha}\tilde{\delta}, \tilde{\beta}\tilde{\gamma}\rangle$ and
$A_2= KQ_2/\langle \alpha_1\alpha_2, \gamma\alpha_1, \alpha_2\beta \rangle$ have the same ribbon graph $\Gamma$ as gentle algebras given in Figure \ref{BG}. Note that the associated Brauer graph algebra $\Lambda_\Gamma$ in Figure \ref{BGAG} is the trivial extension of both $A_1$ and $A_2$. Here, the algebra $A_1$ is non-domestic as the bands $\tilde{\alpha}\tilde{\gamma}\tilde{\delta}^{-1}\tilde{\beta}^{-1}$ and $\tilde{\alpha}\tilde{\beta}^{-1}$ commute. On the other hand, $A_2$ has only one band $\gamma\beta^{-1}$, so it is domestic. Thus $A_1$ has a superdecomposable module whereas $A_2$ has none \cite{Puni04}. We show that the Brauer graph algebra $\Lambda_\Gamma= T(A_1)=T(A_2)$ has a superdecomposable module, as it is a trivial extension algebra of $A_1$ by Theorem \ref{TSDM}. 

Now, consider the $\bar{U}=\alpha_1\beta_1^{-1}$ and $\bar{V}=\alpha_1\alpha_2\beta_2^{-1}\beta_1^{-1}$. Then for any $S, T\in \Sigma(U,V), S', T'\in \Sigma(U^{-1}, V^{-1})$, the set $L^T_S(U,V)$ and $L^{T'}_{S'}(U^{-1},V^{-1})$ are dense chains without endpoints by Theorem \ref{DCWE} and $((M(X), z_1^X)_{X\in L^T_S(U,V)}; (M(Y), z_1^Y)_{Y\in L^{T'}_{S'}(U^{-1},V^{-1})})$ is an independent pair of dense chains of pointed $\Lambda_\Gamma$-modules by Theorem \ref{IPDC}. Note that $\Lambda_\Gamma=T(A_2)$ has a superdecomposable module whereas $A_2$ has none.
\end{example}

We end this section with the next remark stating that a relation extension algebra also preserves a superdecomposable module, as it is a specific type of trivial extension algebra. 
\begin{rmk}
An algebra $\tilde{\Lambda}$ is cluster-tilted if and only if there exists a tilted algebra $\Lambda$ such that $\tilde{\Lambda}$ is a relation extension of $\Lambda$. A relation extension algebra of $C$ is a specific type of trivial extension algebra $C\ltimes \mathrm{Ext}_C^2(DC, C)$ by the C-C-bimodule $\mathrm{Ext}_C^2(DC, C)$. Follow \cite{ABS08} for details. So it's clear from Theorem \ref{TSDM} that the existence of a superdecomposable module over $C$ implies the existence of one over $\tilde{C}$.  
\end{rmk}

\end{document}